\documentclass[10pt]{article}

\usepackage{arxiv}

\usepackage{amsthm}
\usepackage{amsmath}
\usepackage{epstopdf}% To incorporate .eps illustrations using PDFLaTeX, etc.
\usepackage[caption=false]{subfig}% Support for small, `sub' figures and tables
\usepackage{multirow}

\usepackage{amsopn}
\newtheorem{theorem}{Theorem}[section]
\newtheorem{lemma}[theorem]{Lemma}
\newtheorem{corollary}[theorem]{Corollary}
\newtheorem{proposition}[theorem]{Proposition}

\newtheorem{definition}[theorem]{Definition}

\newtheorem{remark}{Remark}

\usepackage[utf8]{inputenc} % allow utf-8 input
\usepackage[T1]{fontenc}    % use 8-bit T1 fonts
\usepackage{hyperref}       % hyperlinks
\usepackage{url}            % simple URL typesetting
\usepackage{booktabs}       % professional-quality tables
\usepackage{amsfonts}       % blackboard math symbols
\usepackage{nicefrac}       % compact symbols for 1/2, etc.
\usepackage{microtype}      % microtypography
\usepackage{lipsum}
\usepackage{fancyhdr}       % header
\usepackage{graphicx}       % graphics
\title{A smoothing analysis for multigrid methods applied to tempered fractional problems
}

\author{
	D. Ahmad\\
	Università degli studi dell'Insubria \\
	Como, Italy\\
	\texttt{dahmad@uninsubria.it} \\
	\And
	M. Donatelli \\
	Università degli studi dell'Insubria \\
	Como, Italy\\
	\texttt{marco.donatelli@uninsubria.it}
		\And
	M. Mazza \\
	Università degli studi dell'Insubria \\
	Como, Italy\\
	\texttt{mariarosa.mazza@uninsubria.it}
		\And
	S. Serra-Capizzano \\
	Università degli studi dell'Insubria \\
	Como, Italy\\
	\texttt{s.serracapizzano@uninsubria.it}
	\And
	K. Trotti \\
	Universit\`a della Svizzera Italiana\\
	Lugano, Switzerland \\
	\texttt{ken.trotti@usi.ch}
}

\begin{document}
	\maketitle

\begin{abstract}
	We consider the numerical solution of time-dependent space tempered fractional diffusion equations. The use of Crank-Nicolson in time and of second-order accurate tempered weighted and shifted Gr\"unwald difference in space leads to dense (multilevel) Toeplitz-like linear systems. By exploiting the related structure, we design an ad-hoc multigrid solver and multigrid-based preconditioners, all with weighted Jacobi as smoother. A new smoothing analysis is provided, which refines state-of-the-art results expanding the set of the suitable Jacobi weights. Furthermore, we prove that if a multigrid method is effective in the non-tempered case, then the same multigrid method is effective also in the tempered one. 
	The numerical results confirm the theoretical analysis, showing that the resulting multigrid-based solvers are computationally effective for tempered fractional diffusion equations.
\end{abstract}

	% keywords can be removed

\section{Introduction}
Tempered fractional derivatives are a generalization of fractional derivatives where an exponential tempering is involved \cite{meer2015}. In case a Riemann-Liouville formulation is adopted, given $\alpha \in (n-1, n)$, $n \in \mathbb{N^{+}}$, $\lambda\geq 0$, we define the left-handed and right-handed Riemann-Liouville tempered fractional derivatives \cite{baeumer2010tempered}, respectively, as
%\begin{definition}[\cite{baeumer2010tempered}]\label{def1}
%     Let $\alpha \in (n-1, n)$, $n \in \mathbb{N^{+}}$, $\lambda\geq 0$ and let $u(x):[a,b]\rightarrow\mathbb{R}$ be smooth enough.
%be $(n-1)$ times continuously differentiable on $(a,\infty)$ (or $(-\infty, b)$ corresponding to the right derivative) and its $n-$times derivative be integrable on any subinterval of 
%     \begin{math}
%	[\, a,\infty)\,
%\end{math} (or \begin{math}
%	(\,-\infty, b]\,
%\end{math} 
%corresponding to the right derivative), $\lambda\geq 0$.
%the left Riemann-Liouville tempered fractional derivative as
\begin{align}\label{def1}
%	\begin{split}
		\ _{a}{\mathcal{D}_{x}^{\alpha,\lambda}}u(x)&=\frac{{\rm{e}}^{-\lambda x}}{\Gamma(n-\alpha)}\frac{d^{n}}{dx^{n}}\int_{a}^{x}\frac{{\rm{e}}^{\lambda \xi}u(\xi)}{(\,x-\xi)\,^{\alpha-n+1}}d{\xi},\\
		\ _{x}{\mathcal{D}_{b}^{\alpha,\lambda}}u(x)&=\frac{(\,-1)\,^{n}{\rm{e}}^{\lambda x}}{\Gamma(n-\alpha)}\frac{d^{n}}{dx^{n}}\int_{x}^{b}\frac{{\rm{e}}^{-\lambda \xi}u(\xi)}{(\,\xi-x)\,^{\alpha-n+1}}d{\xi}, 
%	\end{split}
\end{align}
Here $\Gamma(\cdot)$ is the Euler gamma function.
%\end{definition}
It is clear that for $\lambda=0$ the exponential tempering vanishes and the tempered derivatives reduce to the classical Riemann-Liouville fractional derivatives (see \cite{podlubny1998fractional}).

Tempered Fractional Diffusion Equations (TFDEs) are standard diffusion equations with a tempered fractional derivative in place of a second order one. TFDEs have been found useful to model phenomena in physics, finance, biology and hydrology (see \cite{meer2015} and citations therein). A common approach to numerically solve time-dependent space-TFDEs is to use Crank-Nicolson (CN) in time and Tempered Weighted and Shifted Gr\"unwald Difference (TWSGD) in space, shortly CN-TWSGD method.  
High order CN-TWSGD schemes have been studied in \cite{li2016high, bu2021high}, where the stability depending on some parameters is proved.
%In \cite{ the authors proved the stability of a second order CN-TWSGD scheme, while in \cite{li2016high} the authors proved the stability of a third order CN-TWSGD scheme. 
These schemes lead to sequences of linear systems to be solved at each time step, requiring preconditioning or multigrid strategies for dealing with the large size of the involved full and ill-conditioned coefficient matrices.

Multigrid methods (MGMs) have widely been applied in the field of Fractional Diffusion Equations (FDEs) (see e.g. \cite{donatelli2020multigrid,donatelli2016spectral,moghaderi2017spectral,meerschaert2004finite,ChEkSC}). Conversely, regarding TFDEs, MGMs have not yet extensively been investigated. In the one-dimensional case, in \cite{chen2017} the authors proved the convergence of the V-cycle with a finite element discretization.
%In case of a Finite Difference discretization, in \cite{bu2021high} the authors focused on a third order space discretization, % and provided the stability study in case of a Crank-Nicolson time discretization scheme, 
%while 
Concerning finite difference discretizations, in \cite{bu2021multigrid} relaxed Jacobi as smoother was adopted and the convergence was proven for weights in $(0,\frac{1}{2}]$. 

In this work, for the sake of notational simplicity, we deal with a second-order accurate finite difference discretization of a time-dependent space-TFDE. Nevertheless, the same analysis can also be extended to the higher order schemes studied in \cite{li2016high}. The use of CN-TWSGD leads to a sequence of linear systems with coefficient matrices having a Toeplitz structure. Similarly to the study in \cite{donatelli2016spectral,moghaderi2017spectral,donatelli2020multigrid} for non-tempered FDEs, we provide the analysis of the generating function of the sequence of Toeplitz matrices.
In particular, we prove that such a generating function is the same as in non-tempered case. This allows to extend the numerical proposals and the convergence analysis in \cite{donatelli2016spectral,moghaderi2017spectral} to the tempered case too.
Therefore, the analytic features of the considered function are used to design efficient multigrid-based solvers by extending the theoretical result in \cite{bu2021multigrid}, providing a wider range for the Jacobi relaxation parameter $\omega$. This is of particular interest in the applications because the numerical results in \cite{bu2021multigrid}  prove that the optimum values of $\omega$ is usually larger than $\frac{1}{2}$ and such value is predicted and estimated by our smoothing analysis based on the spectral information deduced by the generating function. Moreover, the Laplacian preconditioner proposed in \cite{donatelli2016spectral,moghaderi2017spectral} is still a good alternative to multigrid methods in particular for $\alpha$ close to $2$.

In the numerical results, we provide evidences that support our choice of the smoother relaxation parameter also in the two-dimensional and steady-state cases.
The two-dimensional space opens the door to possible anisotropies, when the difference in fractional orders $\alpha,\beta$ is large, and ad-hoc robust MGMs are required (see e.g. \cite{donatelli2020multigrid} in the case of FDEs). In this work we only consider the isotropic case $\alpha\approx\beta$. %According to the theoretical analysis, the Laplacian preconditioner is a robust and effective choice when $\alpha\approx 2$.

This paper is organized as follows. Section \ref{sect:TFDE} describes the CN-TWSGD method for space-TFDE. The spectral analysis of the coefficient matrix at each time-step of CN is provided in Section \ref{sect:symb}. In Section \ref{sect:mgm} we propose an alternative to the convergence analysis of multigrid methods for TFDEs given in \cite{bu2021multigrid} providing a theoretical estimation of the Jacobi relaxation parameter. Sections \ref{sect:num}-\ref{sect:2D} are devoted to some 1D and 2D numerical results, respectively. Finally, in Section \ref{sect:concl} we draws conclusions and we discuss open questions and future research lines.

%------------------------------------------------------------------------------------------------	
\section{Tempered fractional diffusion equations and \mbox{CN-TWSGD} scheme}\label{sect:TFDE}
We are interested in the following space-\mbox{TFDE}
\begin{equation}\label{eq:TFDE}
	\begin{cases}
		\frac{\partial u(\,x,t)\,}{\partial t}=c_{l}(\,x,t)\,{_{a}{\mathbb{{D}}_{x}^{\alpha,\lambda}}u(\,x,t)\,}+c_{r}(\,x,t)\,{_{x}{\mathbb{{D}}_{b}^{\alpha,\lambda}}u(\,x,t)\,}+f(\,x,t)\,,\quad (\,x,t)\,\in(\,a,b)\,\times(\,0,T),\,\\
		u(\,a,t)\ =g_{1}(t),\quad u(\,b,t)\,=g_{2}(t),\quad t\in(\,0,T)\,,\\
		u(\,x,0)\,=S(x),\quad x\in(\,a,b)\,,
	\end{cases}
\end{equation}

where $\alpha\in(\,1,2)\,$, $f(\,x,t)\,$ is the source term, $c_{l}(\,x,t)\,,~c_{r}(\,x,t)\,$ are non-negative diffusion coefficients,
and ${_{a}{\mathbb{{D}}_{x}^{\alpha,\lambda}}}$ and ${_{x}{\mathbb{{D}}_{b}^{\alpha,\lambda}}}$ are variants of the left and right Riemann-Liouville tempered fractional derivatives given in the following definition. 
\begin{definition}[\cite{baeumer2010tempered,cartea2007fluid}]\label{def2}
	Let $_{a}{\mathcal{{D}}_{x}^{\alpha,\lambda}}u(x)$ and $_{x}{\mathcal{{D}}_{b}^{\alpha,\lambda}}u(x)$ be defined as in \eqref{def1}, then for  $1<\alpha<2$ we define
	\begin{align*}%\label{eq:aVariants}
		_{a}{\mathbb{{D}}_{x}^{\alpha,\lambda}}u(x) &=
		_{a}{\mathcal{D}_{x}^{\alpha,\lambda}}u(x)-\alpha\lambda^{\alpha-1}\partial_{x}u(\,x)\,-\lambda^{\alpha}u(\,x)\,, \\
		_{x}{\mathbb{{D}}_{b}^{\alpha,\lambda}}u(x) &=
		_{x}{\mathcal{D}_{b}^{\alpha,\lambda}}u(x)+\alpha\lambda^{\alpha-1}\partial_{x}u(\,x)\,-\lambda^{\alpha}u(\,x)\,\\
		_{a}{\mathbb{\boldsymbol{D}}_{x}^{\alpha,\lambda}}u(x)&={_{a}{\mathcal{D}_{x}^{\alpha,\lambda}}u(x)}-\lambda^{\alpha}u(x),\\
		_{x}{\mathbb{\boldsymbol{D}}_{b}^{\alpha,\lambda}}u(x)&={_{x}{\mathcal{D}_{b}^{\alpha,\lambda}}u(x)}-\lambda^{\alpha}u(x),
	\end{align*}
\end{definition}
\begin{remark}\label{remark1}
	In the above definition, $a$ and $b$ can be extended to $-\infty$ and $\infty$ respectively, by smoothly zero extending $u(x)$ to $(\,-\infty, b)\,$ or $(\,a,\infty)\,$ or even $(\,-\infty,\infty)$.
\end{remark}

\subsection{Discretization of tempered fractional derivatives}	
In \cite{meerschaert2004finite}, Meerschaert and Tadjrean proved that the implicit Euler method based on the shifted Gr\"unwald formula is consistent and unconditionally stable, while the standard Gr\"unwald discretization leads to an unstable scheme when it is used to solve time-dependent FDEs. A similar scenario occurs also in case of time-dependent TFDEs, so a tempered counterpart of the shifted Gr\"unwald formula has been introduced \cite{li2016high}. In this subsection, we recall the shifted Gr\"unwald difference operator for the Riemann-Liouville tempered fractional derivatives.
%	We denote , where $a$ and $b$ can be extended to $-\infty$ and $\infty$, respectively.

%\theoremstyle{plain}

\begin{theorem}\label{Theorem1}
	\cite{li2016high} Let $u(x)\in L^{1}(\mathbb{R})$, $_{-\infty}{\mathbb{\boldsymbol{D}}_{x}^{\alpha+l,\lambda}}u(x)$, $_{x}{\mathbb{\boldsymbol{D}}_{\infty}^{\alpha+l,\lambda}}u(x)$ and its Fourier transform belong to $L^{1}(\mathbb{R})$. 
	We define the left and right tempered $\mbox{WSGD}$ operator by 
	\begin{equation}\label{eq:aWSGO}
		{_{L}\mathcal{D}_{h,p_{1},p_{2},\cdots,p_{m}}^{\alpha,\gamma_{1},\gamma_{2},\cdots,\gamma_{m}}}u(x)=\sum_{j=1}^{m}{\gamma_{j}}	{\mathcal{A}_{h,p_{j}}^{\alpha,\lambda}}u(x),
	\end{equation}
	with
	\begin{equation*}		{\mathcal{A}_{h,p_{j}}^{\alpha,\lambda}}u(x)=\frac{1}{h^{\alpha}}\left(\sum_{k=0}^{\lfloor\frac{x-a}{h}\rfloor+p_{j}}\omega_{k}^{\alpha}{{\rm{e}}^{-(k-p_{j})h\lambda}}u(\,x-(\,k-p_{j})\,h)\,-\,{\rm{e}}^{p_{j}h\lambda}(\,1-{\rm{e}}^{-h\lambda})\,^{\alpha}\,u(x)\right),
	\end{equation*}
	and
	\begin{equation}\label{eq:bWSGO}
		{_{R}\mathcal{D}_{h,p_{1},p_{2},\cdots,p_{m}}^{\alpha,\gamma_{1},\gamma_{2},\cdots,\gamma_{m}}}u(x)=\sum_{j=1}^{m}{\gamma_{j}}	{\mathcal{B}_{h,p_{j}}^{\alpha,\lambda}}u(x),
	\end{equation}
	with
	\begin{equation*}
		{\mathcal{B}_{h,p_{j}}^{\alpha,\lambda}}u(x)=\frac{1}{h^{\alpha}}\left(\sum_{k=0}^{\lfloor\frac{b-x}{h}\rfloor+p_{j}}\omega_{k}^{\alpha}{{\rm{e}}^{-(k-p_{j})h\lambda}}u(\,x+(\,k-p_{j})\,h)\,-\,{\rm{e}}^{p_{j}h\lambda}(\,1-{\rm{e}}^{-h\lambda})\,^{\alpha}\,u(x)\right),
	\end{equation*}
\end{theorem}
where $\omega_{k}^{\alpha}=(-1)^{k}\binom{\alpha}{k}$ denote the alternating factional binomial coefficients, while the parameters $p_{j}$ and $\gamma_{j}\in \mathbb{R}$ depend on~$l$. Then, for any integer $m\geq l$
\begin{equation}\label{eq:aWSGOl}
	{_{L}\mathcal{D}_{h,p_{1},p_{2},\cdots,p_{m}}^{\alpha,\gamma_{1},\gamma_{2},\cdots,\gamma_{m}}}u(x)={_{-\infty}{\mathbb{\boldsymbol{D}}_{x}^{\alpha,\lambda}}u(x)}+O(\,h^{l})\,,
\end{equation}	
and 
\begin{equation}\label{eq:bWSGOl}	
	{_{R}\mathcal{D}_{h,p_{1},p_{2},\cdots,p_{m}}^{\alpha,\gamma_{1},\gamma_{2},\cdots,\gamma_{m}}}u(x)={_{x}{\mathbb{\boldsymbol{D}}_{\infty}^{\alpha,\lambda}}u(x)}+O(\,h^{l})\,,
\end{equation}
uniformly for $x\in \mathbb{R}$.\\
In the following, we fix our attention to the second order accurate case, i.e. $l=2$. As a consequence, $p_{j}$ and $\gamma_{j}$ should satisfy the following conditions
\begin{align*}%\label{eq:l2}
	\begin{cases}
		\sum_{j=1}^{m}{\gamma_{j}}=1,\\
		\sum_{j=1}^{m}{\gamma_{j}\left[p_{j}-{\frac{\alpha}{2}}\right]}=0.
	\end{cases}       
\end{align*}
Orders $l>2$ and corresponding resized values of $p_{j}$ and $\gamma_{j}$ can be found in \cite{bu2021high}. For computational purposes and ease of presentation, we are more interested in schemes where $\vert p_{j}\vert\leq 1$. In the following sections, we take $p_{1}=1$, $p_{2}=0$, and $p_{3}=-1$. Note that the change of the order of the shifting parameter does not lead to any modification in the operator. By fixing $\gamma_{3}$, the parameters $\gamma_{1}$ and $\gamma_{2}$ satisfy
\begin{align}\label{eq:g1g2}
	\begin{cases}
		{\gamma_{1}}=\frac{\alpha}{2}+\gamma_{3},\\
		\gamma_{2}=\frac{2-\alpha}{2}-2\gamma_{3}.
	\end{cases}       
\end{align}

\subsection{Second-order CN-TWSGD scheme for the TFDE}\label{subsec:CN-TWSGD}
According to the proposal in \cite{li2016high}, we consider the following equispaced space and time grids over the domain $[a,b]\times[0,T]$:
\begin{align*}
	x_{i} &= a+ih,  &i=1,\cdots,M,  \qquad h&=\frac{b-a}{M+1},\\
	t_{j} &= j\tau,  &j=1,\cdots,N,  \qquad \tau&=\frac{T}{N}.
\end{align*}
Then, by discretizing equation \eqref{eq:TFDE} with CN in time, we obtain the semi-discrete scheme

%	Let the equidistant time partition, $t_{j} = j\tau$ $(0\leq t_{j}\leq T, j=0,1,\cdots,N)$ and spatial grid $x_{i} = a+ih$ $(a\leq x_{i}\leq b, i=0,1,\cdots,M+1)$ be defined, where $\tau=\frac{T}{N}$ and $h=\frac{b-a}{M+1}$. Using the left and right variants of Riemann-Liouville tempered fractional derivatives, we obtain the following Crank–Nicolson time discretization for equation \eqref{eq:TFDE} at mesh point $(x_{i} , t_{j})$:

\begin{equation*}\label{eq:CN}
	\begin{split}
		\frac{ u_{i}^{j+1}-u_{i}^{j}}{\tau}\,=\,&c_{l,i}^{j+\frac{1}{2}}{(\,_{a}{\mathbb{\boldsymbol{D}}_{x}^{\alpha,\lambda}}u_{i}^{j+\frac{1}{2}}-\alpha\lambda^{\alpha-1}\partial_{x}u_{i}^{j+\frac{1}{2}})\,}+c_{r,i}^{j+\frac{1}{2}}{(\,_{x}{\mathbb{\boldsymbol{D}}_{b}^{\alpha,\lambda}}u_{i}^{j+\frac{1}{2}}+\alpha\lambda^{\alpha-1}\partial_{x}u_{i}^{j+\frac{1}{2}})\,}+f_{i}^{j+\frac{1}{2}}\\&+O(\tau^{2}+h^{2}),
	\end{split}
\end{equation*}    
where $u_{i}^{j}=u(x_{i},t_{j})$, $c_{l,i}^{j}=c_{l}(x_{i},t_{j})$, $c_{r,i}^{j}=c_{r}(x_{i},t_{j})$ and $u_{i}^{j+\frac{1}{2}}=\frac{u_{i}^{j}+u_{i}^{j+1}}{2}$. Replacing the tempered fractional derivatives with the TWSGD operators given in equations \eqref{eq:aWSGO} and \eqref{eq:bWSGO}, we obtain the CN-TWSGD scheme
\begin{equation}\label{eq:CN_WSGD}
	\begin{split}
		&{u_{i}^{j+1}}-\frac{\tau}{2}\left[c_{l,i}^{j+1}\left({_{L}\mathcal{D}_{h,1,0,-1}^{\alpha,\gamma_{1},\gamma_{2},\gamma_{3}}}u_{i}^{j+1}\right)+c_{r,i}^{j+1}\left({_{R}\mathcal{D}_{h,1,0,-1}^{\alpha,\gamma_{1},\gamma_{2},\gamma_{3}}}u_{i}^{j+1}\right)-\alpha\lambda^{\alpha-1}{\left(c_{l,i}^{j+1}-c_{r,i}^{j+1}\right)\left(\frac{u_{i+1}^{j+1}-u_{i-1}^{j+1}}{2h}\right)}\right]\\&
		=u_{i}^{j}+\frac{\tau}{2}\left[c_{l,i}^{j}\left({_{L}\mathcal{D}_{h,1,0,-1}^{\alpha,\gamma_{1},\gamma_{2},\gamma_{3}}}u_{i}^{j}\right)+c_{r,i}^{j}\left({_{R}\mathcal{D}_{h,1,0,-1}^{\alpha,\gamma_{1},\gamma_{2},\gamma_{3}}}u_{i}^{j}\right)-\alpha\lambda^{\alpha-1}{\left(c_{l,i}^{j}-c_{r,i}^{j}\right)\left(\frac{u_{i+1}^{j}-u_{i-1}^{j}}{2h}\right)}\right]\\&+\tau f_{i}^{j+\frac{1}{2}}+O(\tau^{3}+\tau h^{2}).
	\end{split}
\end{equation}

After rearranging the weights $\omega_{k}^{\alpha}$ in equation \eqref{eq:aWSGOl}-\eqref{eq:bWSGOl}, the Riemann-Liouville tempered fractional derivatives at point $x_{i}$ are approximated as
\begin{align*}%\label{eq:rearrangedWSGD}
	\begin{split}
		{_{a}{\mathcal{D}_{x}^{\alpha,\lambda}}u(x_{i})}-\lambda^{\alpha}u(x_{i})&=\frac{1}{h^{\alpha}}\left[\sum_{k=0}^{i+1}g_{k,\lambda}^{(\,2,\alpha)\,}u(\,x_{i-k+1})\,-\phi(\,\lambda)\,u(x_{i})\right]+O{(\,h^{2})\,},\\
		{_{x}{\mathcal{D}_{b}^{\alpha,\lambda}}u(x_{i})}-\lambda^{\alpha}u(x_{i})&=\frac{1}{h^{\alpha}}\left[\sum_{k=0}^{M-i+1}g_{k,\lambda}^{(\,2,\alpha)\,}u(\,x_{i+k-1})\,-\phi(\,\lambda)\,u(x_{i})\right]+O{(\,h^{2})\,},
	\end{split}
\end{align*}
where 
\begin{equation}\label{eq:phi}
	\phi(\,\lambda)\,=(\,\gamma_{1}{\rm{e}}^{h\lambda}+\gamma_{2}+\gamma_{3}{\rm{e}}^{-h\lambda})\,(\,1-{\rm{e}}^{-h\lambda})\,^{\alpha}=\sum_{j=1}^{m=3}{\gamma_{j}{\rm{e}}^{p_{j}h\lambda}}{(\,1-{\rm{e}}^{-h\lambda})\,^{\alpha}} 
\end{equation}
and %the weights $\gamma_1$, $\gamma_2$, and $\gamma_3$ are given by the linear system
\begin{align*}%\label{eq:weights}
	\begin{cases}
		{g_{0,\lambda}^{(\,2,\alpha)\,}}=\gamma_{1}\omega_{0}^{\alpha}{\rm{e}}^{h\lambda},\\
		{g_{1,\lambda}^{(\,2,\alpha)\,}}=\gamma_{1}\omega_{1}^{\alpha}+\gamma_{2}\omega_{0}^{\alpha},\\
		{g_{k,\lambda}^{(\,2,\alpha)\,}}=\left(\gamma_{1}\omega_{k}^{\alpha}+\gamma_{2}\omega_{k-1}^{\alpha}+\gamma_{3}\omega_{k-2}^{\alpha}\right){\rm{e}}^{-(k-1)h\lambda},~k\geq 2.
	\end{cases}       
\end{align*}

Using the following notations, $u^{j}=[\,u_{1}^{j},u_{2}^{j},\cdots,u_{M}^{j}]\,^{T}, c_{l}^{j}=\mbox{diag}(\,c_{l,1}^{j},c_{l,2}^{j},\cdots,c_{l,M}^{j})\,$, $c_{r}^{j}=\mbox{diag}(\,c_{r,1}^{j},c_{r,2}^{j},\cdots,c_{r,M}^{j})\,$, and

\begin{equation}\label{eq:B_matrix}
	B_{M,\lambda}^{\left(2,\alpha\right)}=
	\begin{bmatrix}
		{g_{1,\lambda}^{(\,2,\alpha)\,}}-\phi(\,\lambda)\, & {g_{0,\lambda}^{(\,2,\alpha)\,}} & 0 & 0 & \cdots & 0 \\
		{g_{2,\lambda}^{(\,2,\alpha)\,}} & {g_{1,\lambda}^{(\,2,\alpha)\,}}-\phi(\,\lambda)\, & {g_{0,\lambda}^{(\,2,\alpha)\,}} & 0 & \cdots & \vdots \\
		\vdots & \ddots & \ddots & \ddots &\ddots & \vdots\\
		\vdots & \ddots & \ddots &\ddots &\ddots & \vdots\\
		{g_{M-1,\lambda}^{(\,2,\alpha)\,}} & {g_{M-2,\lambda}^{(\,2,\alpha)\,}} & \cdots & \cdots & {g_{1,\lambda}^{(\,2,\alpha)\,}}-\phi(\,\lambda)\, & {g_{0,\lambda}^{(\,2,\alpha)\,}} \\
		{g_{M,\lambda}^{(\,2,\alpha)\,}} & {g_{M-1,\lambda}^{(\,2,\alpha)\,}} &\cdots & \cdots & {g_{2,\lambda}^{(\,2,\alpha)\,}} & {g_{1,\lambda}^{(\,2,\alpha)\,}}-\phi(\,\lambda)\, 
	\end{bmatrix}_{{M\times M}},
\end{equation}
the corresponding matrix form of equation \eqref{eq:CN_WSGD}, neglecting the remainder, can be written as
\begin{equation}\label{eq:Matrix_form}
	\left(I_M-\mathcal{M}_{M}^{j+1}\right)u^{j+1}=\left(I_M+\mathcal{M}_{M}^j\right)u^{j}+\tau{F^{j+1}},
\end{equation}
where $I_M$ is the identity of size $M$, and
\begin{align*}
	\mathcal{M}_{M}^j=\frac{\tau}{2h^{\alpha}}\left(c_{l}^j{B_{M,\lambda}^{\left(2,\alpha\right)}}+c_{r}^j(B_{M,\lambda}^{\left(2,\alpha\right)})^{T}\right)-\frac{\alpha\tau\lambda^{\alpha-1}}{4h}\left(c_{l}^j-c_{r}^j\right)H_{2,M}
\end{align*}
with $H_{2,M}=\mbox{tridiag}\left\{-1,0,1\right\}$ of size ${M}$, and
\[{F^{j+1}}=\begin{bmatrix}
	f_{1}^{j+\frac{1}{2}}\\
	f_{2}^{j+\frac{1}{2}}\\
	\vdots\\
	f_{M-1}^{j+\frac{1}{2}}\\
	f_{M}^{j+\frac{1}{2}}\\
\end{bmatrix}+\frac{{u_{0}^{j}}+{u_{0}^{j+1}}}{2h^{\alpha}}
\begin{bmatrix}
	{c_{l,1}^{j+\frac{1}{2}}}{g_{2,\lambda}^{(\,2,\alpha)\,}}+{c_{r,1}^{j+\frac{1}{2}}}{g_{0,\lambda}^{(\,2,\alpha)\,}}\\
	{c_{l,2}^{j+\frac{1}{2}}}{g_{3,\lambda}^{(\,2,\alpha)\,}}\\
	\vdots\\
	{c_{l,M-1}^{j+\frac{1}{2}}}{g_{M-1,\lambda}^{(\,2,\alpha)\,}}\\
	{c_{l,M}^{j+\frac{1}{2}}}{g_{M,\lambda}^{(\,2,\alpha)\,}}\\
\end{bmatrix}+\frac{{u_{M}^{j}}+{u_{M}^{j+1}}}{2h^{\alpha}}
\begin{bmatrix}
	{c_{r,1}^{j+\frac{1}{2}}}{g_{M,\lambda}^{(\,2,\alpha)\,}}\\
	{c_{r,2}^{j+\frac{1}{2}}}{g_{M-1,\lambda}^{(\,2,\alpha)\,}}\\
	\vdots\\
	{c_{r,M-1}^{j+\frac{1}{2}}}{g_{3,\lambda}^{(\,2,\alpha)\,}}\\
	{c_{r,M}^{j+\frac{1}{2}}}{g_{2,\lambda}^{(\,2,\alpha)\,}}+{c_{l,1}^{j+\frac{1}{2}}}{g_{0,\lambda}^{(\,2,\alpha)\,}}\\
\end{bmatrix}.\]
By defining
\begin{align}\label{eq:Coe_Matrix_form}
	{\mathcal{A}_{M}^{j+1}}=I_{M}-\mathcal{M}_{M}^{j+1},
\end{align}
\begin{align*}%\label{eq:right_Vec}
	f_{M}^{j+1}=\left(I_{M}+\mathcal{M}_{M}^{j+1}\right)u^{j}+\tau{F^{j+1}}.
\end{align*}
the linear system \eqref{eq:Matrix_form}, which has to be solved at each time step $j+1$, can be written as 
\begin{align*}%\label{eq:Final_Linear_sys}
	{\mathcal{A}_{{M}}^{j+1}}{u^{j+1}}=f_{{M}}^{j+1}.
\end{align*}
Using the following Lemma \ref{Lemma1}, it has been proven in \cite{varga2000matrix} that the coefficient matrix is ${\mathcal{A}_{{M}}^{j+1}}$ is strictly diagonally dominant and hence invertible.
\begin{lemma}\label{Lemma1}
	\cite{li2016high} For $1<\alpha<2 ~\mbox{and}~ \lambda\geq0$, it holds
	\begin{equation*}
		\omega_{0}^{\alpha}=1,~\omega_{1}^{\alpha}=-\alpha<0, \qquad {\omega_{0}^{\alpha}>\omega_{2}^{\alpha}>\omega_{3}^{\alpha}>\cdots>0},\qquad \sum_{k=0}^{\infty}{\omega_{k}^{\alpha}}=0.
	\end{equation*}
	Moreover, for $h> 0$ and $$\max\left({\frac{(\,2-\alpha)\,(\,\alpha^{2}+\alpha-8)\,}{2(\,\alpha^{2}+3\alpha+2)\,},\frac{(\,1-\alpha)\,(\,\alpha^{2}+2\alpha)\,}{2(\,\alpha^{2}+3\alpha+4)\,}}\right)\leq\gamma_{3}\leq \frac{(\,2-\alpha)\,(\,\alpha^{2}+2\alpha-3)\,}{2(\,\alpha^{2}+3\alpha+2)\,},$$ then it holds
	\begin{equation*}%\label{eq:Lemma1}
		\begin{cases}
			g_{1,\lambda}^{(\,2,\alpha)\,}\leq0,~ \left(g_{2,\lambda}^{(\,2,\alpha)\,}+g_{0,\lambda}^{(\,2,\alpha)\,}\right)\geq0,\\
			g_{k,\lambda}^{(\,2,\alpha)\,}\geq0,\quad\mbox{for} ~k\geq3, 
		\end{cases}
	\end{equation*} 
	and $$ \sum_{k=0}^{\infty}{g_{k,\lambda}^{(\,2,\alpha)\,}}=\phi(\lambda).$$
\end{lemma}

%%%%%%%%%%%
\section{ Spectral analysis of the coefficient matrix}\label{sect:symb}
This section is devoted to the study of the spectral properties of the coefficient matrix-sequence $\{\mathcal{A}_{{M}}^{j+1}\}_{M\in\mathbb{N}}$. In case of constant diffusion coefficients, the coefficient matrix-sequence is a well-known Toeplitz sequence. We then determine its generating function and study its spectral distribution using spectral tools for Toeplitz sequences. In particular, we prove that the spectral symbol coincides with the generating function. To this aim, let us first introduce some basic definitions and results related to the generating function of a Toeplitz sequence.
\begin{definition}\label{def3}
	\cite{chan1991toeplitz} Let  $T_{M}\in\mathbb{C}^{M\times M}$ be the Toeplitz matrix of the form
	\begin{equation}\label{eq:Toeplitz_M}
		T_{M}=\begin{bmatrix}
			& b_{0} & b_{-1} & b_{-2} & \cdots & \cdots & b_{1-M} &\\
			& b_{1} & b_{0} & b_{-1} & \cdots & \cdots & b_{2-M} &\\
			& \vdots & \ddots & \ddots & \ddots &\ddots & \vdots&\\
			& \vdots & \ddots & \ddots &\ddots &\ddots & \vdots&\\
			& b_{M-2} & \ddots & \ddots & \ddots & \ddots & b_{-1} &\\
			& b_{M-1} & b_{M-2} &\cdots & \cdots & b_{1} & b_{0} &
		\end{bmatrix}
	\end{equation}	
	%	If the diagonals $\left\{{b_{k}}\right\}_{k=-M+1}^{M-1}$ are , $i.e.\,$
	with \begin{equation}\label{eq:fourier_coef}
		b_{k}=\frac{1}{2\pi}\int_{-\pi}^{\pi}{f(x){\rm{e}}^{-\iota kx}dx},\quad\iota^{2}=-1,~~k\in\mathbb{Z},
	\end{equation}
	the Fourier coefficients of a function $f\in L^1(-\pi,\pi)$.
	Then the Toeplitz sequence $\left\{{T_{M}}\right\}_{M\in \mathbb{N}}$ %with ${T_{M}}=\left[b_{i-j}\right]_{i,j=1}^{M}$ 
	is called the sequence of Toeplitz matrices generated by $f$ and the matrix $T_{M}$ in \eqref{eq:Toeplitz_M} is denoted by $T_{M}\left(f\right)$. The function $f$ is called the generating function, both of whole sequence of matrices and of the single matrix $T_{M}\left(f\right)$.
\end{definition}

Note that given a Toeplitz matrix $T_{M}$ as in \eqref{eq:Toeplitz_M}, in order to have a generating function associated to the Toeplitz sequence, we need that there exists $f\in L^1(-\pi,\pi)$ for which the relationship (\ref{eq:fourier_coef}) holds for every $k\in\mathbb{Z}$. In the case where the partial Fourier sum
\[
\sum_{k=-M+1}^{M-1}{b_{k}{\rm{e}}^{\iota kx}}
\]
converges to $f$ when $M\to \infty$ in infinity norm, then $f$ is a continuous $2\pi$ periodic function given the Banach structure of this space. A sufficient condition is that $\sum_{k=-\infty}^{\infty}\vert{b_{k}}\vert<\infty$, i.e., the generating function belongs to the Wiener class, which is a closed sub-algebra of the continuous $2\pi$ periodic functions.

% 	\begin{definition}\label{def4}
%      The \emph{Wiener class} is the set of functions 
% 	\begin{equation*}
% 		f(x)=\sum_{k=-\infty}^{\infty}{b_{k}{\rm{e}}^{\iota kx}},\quad\mbox{s.t.}\quad \left\{b_{k}\right\}_{k\in\mathbb{Z}}\in l^{1}\left(\mathbb{Z}\right),~ i.e.,\, \sum_{k=-\infty}^{\infty}\vert{b_{k}}\vert<\infty.
% 	\end{equation*}
% As already mentioned the Wiener class forms a sub-algebra of the continuous and $2\pi$ periodic functions.  
% 	\end{definition}

Now, according to \eqref{eq:B_matrix}, we define
\begin{equation*}%\label{eq:generating_ftn_Ba}
	b^{\alpha}_{M,\lambda}\left(x\right)=\sum_{k=0}^{M}{g_{k,\lambda}^{(2,\alpha)}{\rm{e}}^{\iota (k-1)x}-\phi(\lambda)}, 
\end{equation*}
with $\phi$ as in \eqref{eq:phi}, and prove the following result.

\begin{proposition}\label{Proposition1}Let $1<\alpha< 2$. The generating function associated to the matrix-sequence $\left\{B_{M,\lambda}^{\left(2,\alpha\right)}\right\}_{M\in\mathbb{N}}$, where $B_{M,\lambda}^{\left(2,\alpha\right)}$ is defined in \eqref{eq:B_matrix}, belongs to the Wiener class.
\end{proposition}
\begin{proof}
	Let us observe that $B_{M,\lambda}^{\left(2,\alpha\right)}=\left[b_{i-j+1}\right]_{i,j=1}^{M}$ with $b_{1}={g_{1,\lambda}^{(\,2,\alpha)\,}}-\phi(\,\lambda)\,$, $b_{k}={g_{k,\lambda}^{(\,2,\alpha)\,}}$ for $k\neq1$ and $b_{k}=0$ for $k<0$. 
	To prove that
	\begin{equation*}
		b^{\alpha}_{\lambda}\left(x\right)=\sum_{k=-1}^{\infty}{b_{k+1}{\rm{e}}^{\iota kx}}
	\end{equation*}
	lies in Wiener class for $\alpha\in \left(1,2\right)$, we have to prove that $\sum_{k=-1}^{\infty}\vert{b_{k+1}}\vert<\infty$. From Lemma \ref{Lemma1}, we know that $b_{1}$ and $b_{2}$ are negative and $b_{k}\geq0$ for $k\geq3$. Then
	\begin{equation*}
		\sum_{k=-1}^{\infty}\vert{b_{k+1}}\vert=\sum_{\substack{k=-1 \\ k\neq 0,1}}^{\infty}{b_{k+1}}+\vert{b_{1}}\vert+\vert{b_{2}}\vert.
	\end{equation*}
	By Lemma \ref{Lemma1}, we have
	\begin{equation*}
		\sum_{k=0}^{\infty}{b_{k}}=0\iff \sum_{\substack{k=-1 \\ k\neq 0,1}}^{\infty}{b_{k+1}}=-\left({b_{1}}+{b_{2}}\right)=\vert{b_{1}}+{b_{2}}\vert,
	\end{equation*}
	and hence
	\begin{align*}
		\sum_{k=-1}^{\infty}\vert{b_{k+1}}\vert&=\vert{b_{1}}+{b_{2}}\vert+\vert{b_{1}}\vert+\vert{b_{2}}\vert,\nonumber\\
		&\leq2\left(\vert{b_{1}}\vert+\vert{b_{2}}\vert\right),
	\end{align*}
	which implies that $b^{\alpha}_{\lambda}$ belongs to the Wiener class.
\end{proof}

After defining 
\begin{equation}\label{eq:BM}
	{\mathcal{\bf B}_{M,\lambda}^{\left(2,\alpha\right)}}=\frac{{B_{M,\lambda}^{\left(2,\alpha\right)}}+(B_{M,\lambda}^{\left(2,\alpha\right)})^{T}}{2},
\end{equation} 
we introduce
\begin{equation*}%\label{eq:generating_ftn}
	f_{M}\left(\alpha,\lambda;x\right)=\frac{	b^{\alpha}_{M,\lambda}\left(x\right)+\overline{	b^{\alpha}_{M,\lambda}\left(x\right)}}{2}=
	\sum_{k=0}^{M}{g_{k,\lambda}^{(2,\alpha)}\cos\left((k-1)x\right)-\phi(\lambda)}.
\end{equation*}
We can now compute the generating function of the Toeplitz sequence $\left\{\mathcal{\bf B}_{M,\lambda}^{\left(2,\alpha\right)}\right\}_{M\in\mathbb{N}}$ proving that is independent of $\lambda$.

\begin{proposition}\label{Proposition2}
	%Let us assume that $c_l(x,t)=c_r(x,t)=c$, then 
	The matrix-sequence $\left\{{\mathcal{\bf B}_{M,\lambda}^{\left(2,\alpha\right)}}\right\}_{M\in\mathbb{N}}$ is generated by the function
	\begin{equation}\label{eq:symf}
		f\left(\alpha;x\right)=\Big(2\sin{\frac{x}{2}}\Big)^{\alpha}\Big(z_{\alpha}(x)+2\gamma_{3}\cos{(\frac{\alpha}{2}(x-\pi))}(\cos{x}-1)\Big),
	\end{equation}
	where
	\[
	z_{\alpha}(x)=\frac{\alpha}{2}\cos{(\frac{\alpha}{2}(x-\pi)-x)}+\frac{2-\alpha}{2}\cos{(\frac{\alpha}{2}(x-\pi))}.
	\]
	%$f_M(\alpha,\lambda;x)$ defined in \eqref{eq:generating_ftn}, we have
	%\begin{align*}
	%    \Big\{{\mathcal{\bf B}_{M,\lambda}^{\left(2,\alpha\right)}}\Big\}_{M\in\mathbb{N}}\sim _{\lambda}\left(c\cdot 	f\left(\alpha;x\right),[0,\pi]\right),
	%\end{align*}
	%\gred{(We have a problem with $\lambda$, it means eigenvalue and it is also the tempering parameter.)\\}
\end{proposition}
\begin{proof}
	%Before going towards the proof of Preposition 2, 
	%Let us first discuss the asymptotical behavior of the generating functions of tempered and non-tempered \mbox{FDE}. For this, we consider the formal expression defined as
	According to \eqref{eq:BM}, the generating function of the matrix-sequence $\left\{{\mathcal{\bf B}_{M,\lambda}^{\left(2,\alpha\right)}}\right\}_{M\in\mathbb{N}}$ is
	\[
	f\left(\alpha,\lambda;x\right) = \frac{b^{\alpha}_{\lambda}\left(x\right) + \overline{b^{\alpha}_{\lambda}\left(x\right)}}{2}\,.
	\]
	Therefore, thanks to Proposition \ref{Proposition1}, $f\left(\alpha,\lambda;x\right)$
	belongs to the Wiener class and hence %there exists $f\left(\alpha,\lambda;x\right)$ such that 
	\begin{equation*}
		\begin{split}
			f\left(\alpha,\lambda;x\right)&=\lim_{M\to\infty}f_M\left(\alpha,\lambda;x\right)\\
			&=\lim_{M\to\infty}\sum_{k=0}^{M}{g_{k,\lambda}^{(2,\alpha)}\Bigg[\frac{{\rm{e}}^{\iota (k-1)x}+{\rm{e}}^{-\iota (k-1)x}}{2}\Bigg]}-\phi(\lambda)\\
			&=\frac{1}{2}\lim_{M\to\infty}\Bigg[g_{0,\lambda}^{(2,\alpha)}({\rm{e}}^{\iota x}+{\rm{e}}^{-\iota x})+g_{1,\lambda}^{(2,\alpha)}+\sum_{k=2}^{M}{g_{k,\lambda}^{(2,\alpha)}\big[{{\rm{e}}^{\iota (k-1)x}+{\rm{e}}^{-\iota (k-1)x}}\big]}\Bigg]-\phi(\lambda)\\
			&=\frac{1}{2}\lim_{M\to\infty}\Bigg[\gamma_{1}\omega_{0}^{\alpha}{\rm{e}}^{h\lambda-\iota x}+\gamma_{1}\omega_{1}^{\alpha}+\gamma_{2}\omega_{0}^{\alpha}+\sum_{k=2}^{M}{\left(\gamma_{1}\omega_{k}^{\alpha}+\gamma_{2}\omega_{k-1}^{\alpha}+\gamma_{3}\omega_{k-2}^{\alpha}\right){\rm{e}}^{-(k-1)h\lambda}}{{{\rm{e}}^{\iota (k-1)x}}}+\\
			&\quad+\gamma_{1}\omega_{0}^{\alpha}{\rm{e}}^{h\lambda+\iota x}+\gamma_{1}\omega_{1}^{\alpha}+\gamma_{2}\omega_{0}^{\alpha}+\sum_{k=2}^{M}{\left(\gamma_{1}\omega_{k}^{\alpha}+\gamma_{2}\omega_{k-1}^{\alpha}+\gamma_{3}\omega_{k-2}^{\alpha}\right){\rm{e}}^{-(k-1)h\lambda}}{{{\rm{e}}^{-\iota (k-1)x}}}\Bigg]-\phi(\lambda)\\
			&=\frac{1}{2}\lim_{M\to\infty}\Bigg[\gamma_{1}{\rm{e}}^{h\lambda-\iota x}\sum_{k=0}^{M}{\omega_{k}^{\alpha}{\rm{e}}^{k(\iota x-h\lambda)}}+\gamma_{2}\sum_{k=0}^{M}{\omega_{k}^{\alpha}{\rm{e}}^{k(\iota x-h\lambda)}}+\gamma_{3}{\rm{e}}^{-h\lambda+\iota x}\sum_{k=0}^{M}{\omega_{k}^{\alpha}{\rm{e}}^{k(\iota x-h\lambda)}}+\\
			&\quad+\gamma_{1}{\rm{e}}^{h\lambda+\iota x}\sum_{k=0}^{M}{\omega_{k}^{\alpha}{\rm{e}}^{k(-\iota x-h\lambda)}}+\gamma_{2}\sum_{k=0}^{M}{\omega_{k}^{\alpha}{\rm{e}}^{k(-\iota x-h\lambda)}}+\gamma_{3}{\rm{e}}^{-h\lambda-\iota x}\sum_{k=0}^{M}{\omega_{k}^{\alpha}{\rm{e}}^{k(-\iota x-h\lambda)}}\Bigg]-\phi(\lambda).
		\end{split}
	\end{equation*}
	Now, recalling that $\omega_{k}^{\alpha}=(-1)^{k}\binom{\alpha}{k}$ and by using the well known binomial series,
	\begin{align*}
		(1+z)^{\alpha}=\sum_{k=0}^{\infty}\binom{\alpha}{k}z^{k},\quad z\in\mathbb{C},\quad\alpha>0,~\vert z\vert\leq 1,
	\end{align*}
	we obtain
	\begin{align*}
		f\left(\alpha,\lambda;x\right)=&\frac{\gamma_{1}}{2}\left[{\rm{e}}^{-\iota x}\left(1-{\rm{e}}^{\iota x}\right)^{\alpha}+{\rm{e}}^{\iota x}\left(1-{\rm{e}}^{-\iota x}\right)^{\alpha}\right]+\frac{\gamma_{2}}{2}\left[\left(1-{\rm{e}}^{\iota x}\right)^{\alpha}+\left(1-{\rm{e}}^{-\iota x}\right)^{\alpha}\right]\nonumber\\
		&+\frac{\gamma_{3}}{2}\left[{\rm{e}}^{\iota x}\left(1-{\rm{e}}^{\iota x}\right)^{\alpha}+{\rm{e}}^{-\iota x}\left(1-{\rm{e}}^{-\iota x}\right)^{\alpha}\right],
	\end{align*}
	which loses the dependency on $\lambda$ that can then be omitted. Finally, through the relation 
	$1-{\rm{e}}^{\iota x}=2\sin{\frac{x}{2}}{\rm{e}}^{\iota \frac{x-\pi}{2}}$
	and by replacing $\gamma_{1}$ and $\gamma_{2}$ given in \eqref{eq:g1g2}, we obtain the thesis.
\end{proof}

As a confirmation of Proposition \ref{Proposition2}, we fix $\alpha=1.5,~\lambda=3$ and $\gamma_{3}=0.01$ and in Figure~\ref{fig: asymptotical behaviour} we depict the functions ${f_{M}\left(\alpha,\lambda;x\right)}$ and ${{f\left(\alpha;x\right)}}$ on $[-\pi, \pi]$. We clearly see that as $M$ increases the two plots overlap.
\begin{figure}
	\centering
	%	\begin{center}$
	\begin{subfloat}[$M=500.$]
		{\resizebox*{4.7cm}{!}{\includegraphics[width=\textwidth]{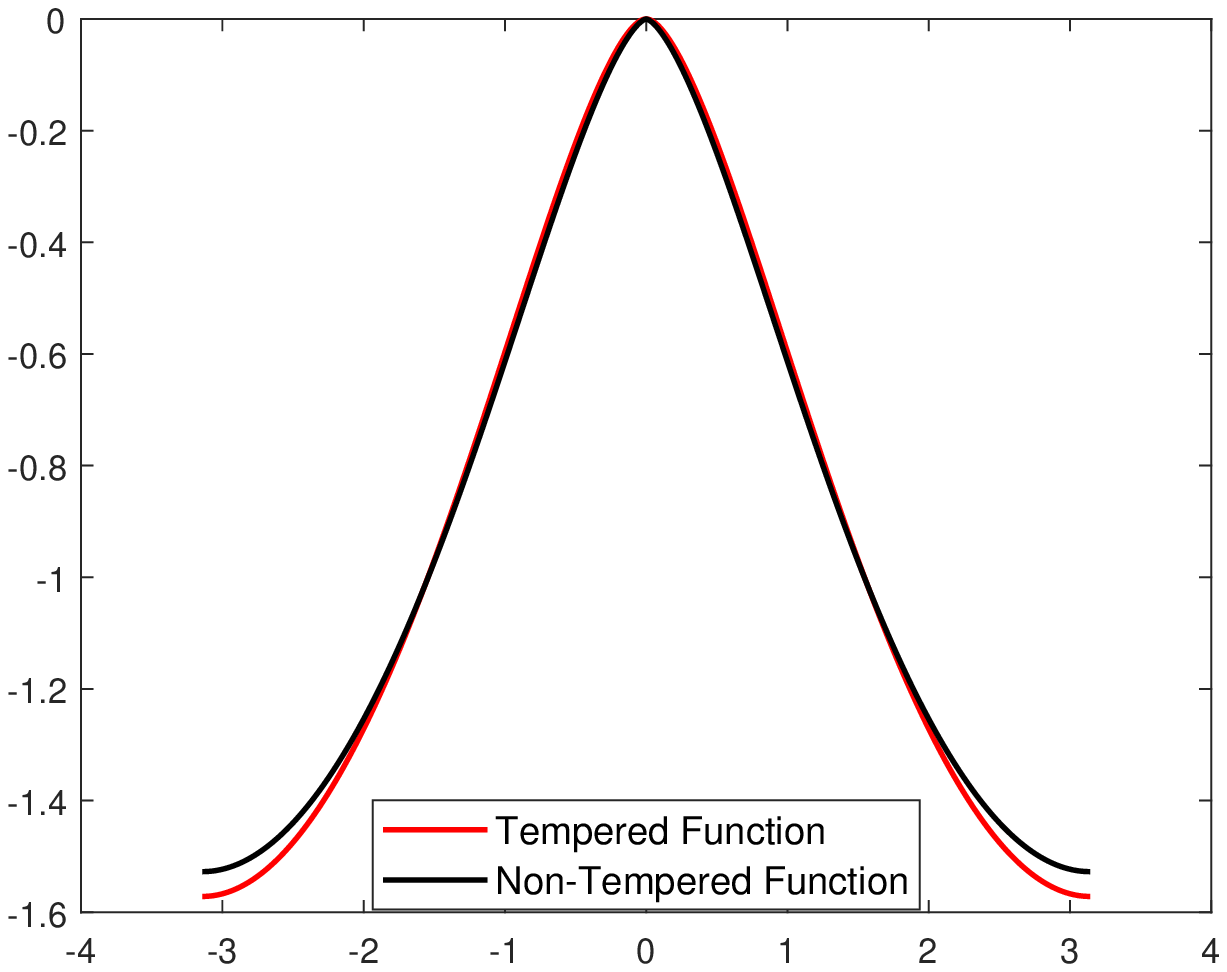}}}
	\end{subfloat}
	\begin{subfloat}[$M=1000.$]
		{\resizebox*{4.7cm}{!}{\includegraphics[width=\textwidth]{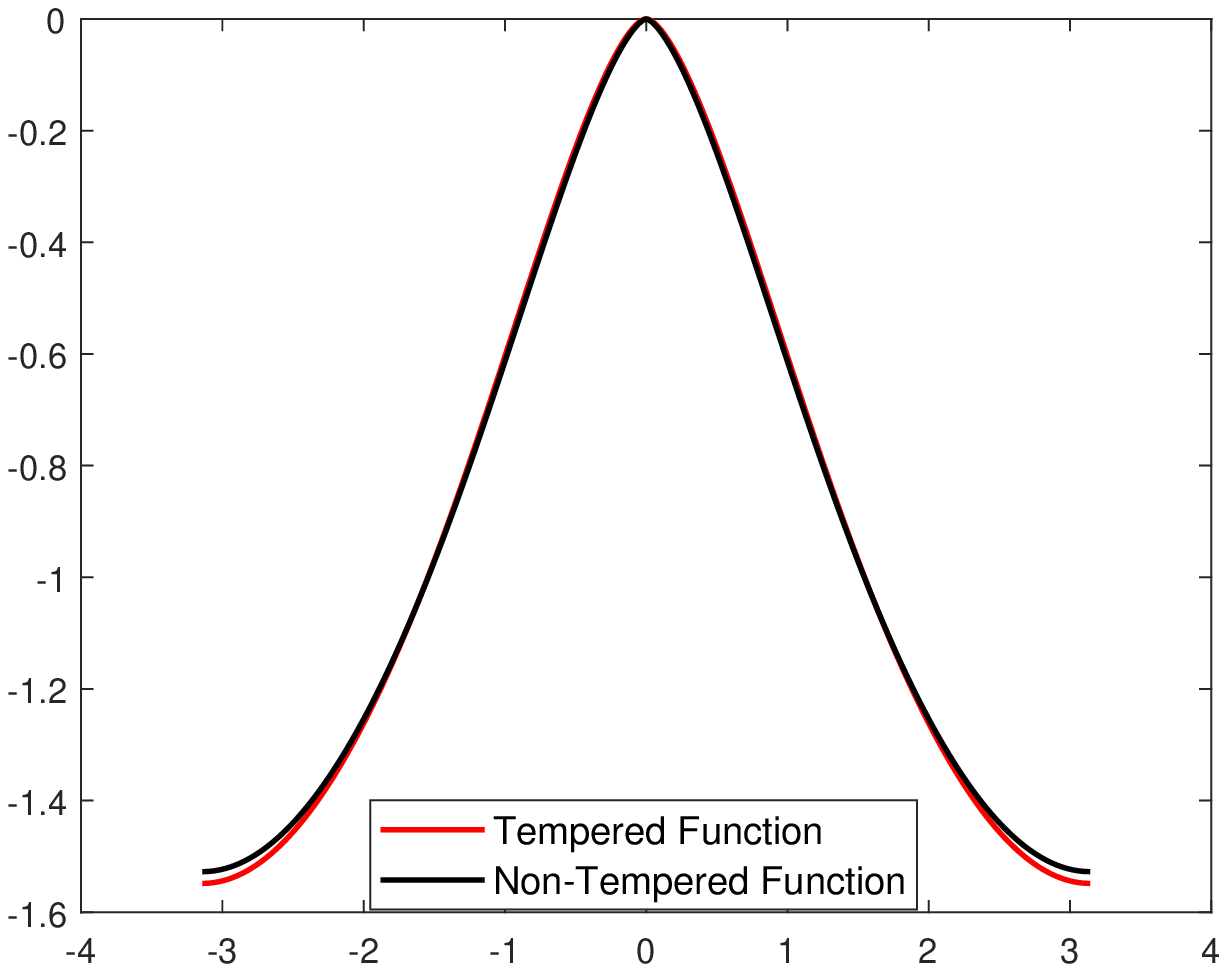}}}
	\end{subfloat}
	\begin{subfloat}[$M=5000.$]
		{\resizebox*{4.7cm}{!}{\includegraphics[width=\textwidth]{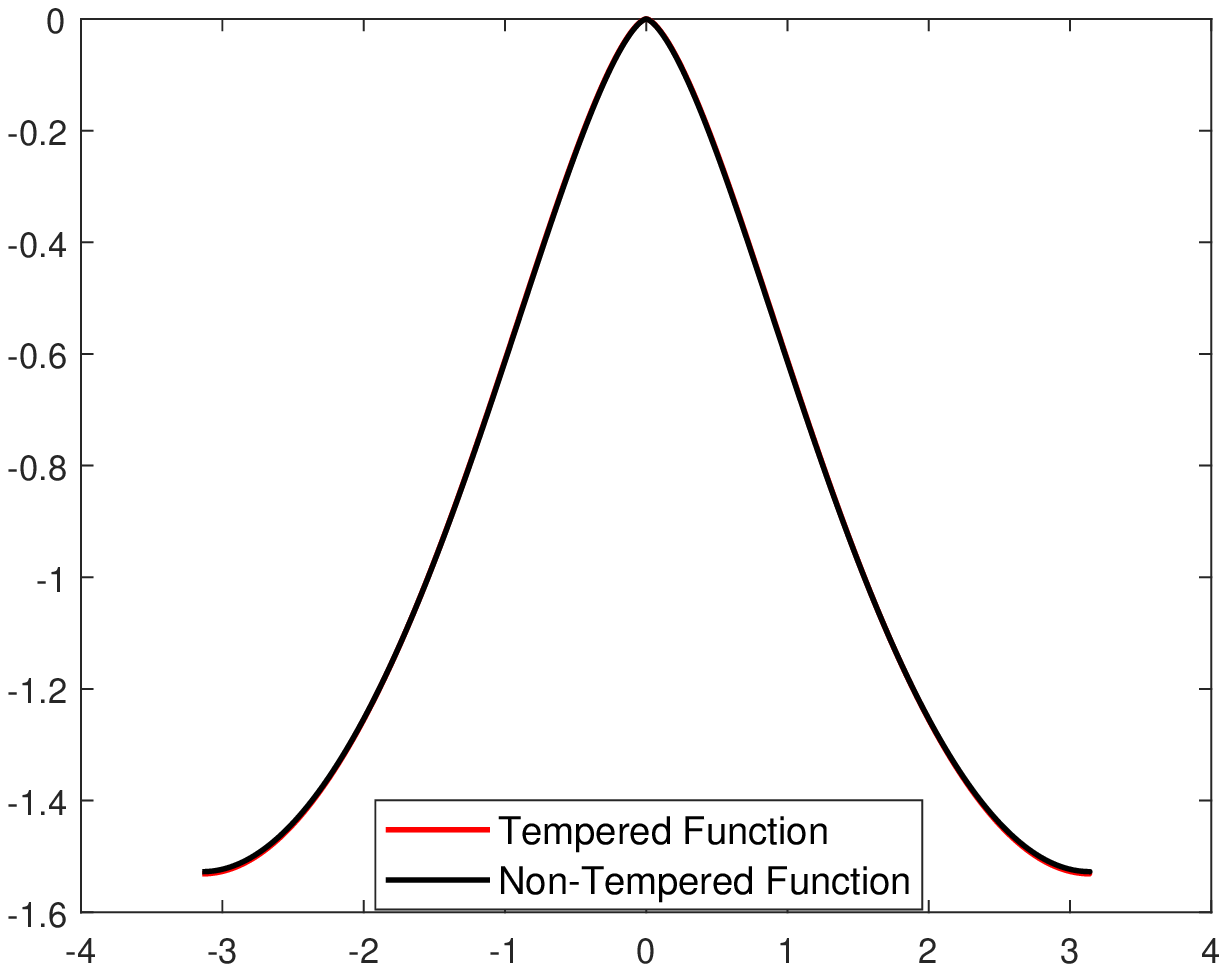}}}
	\end{subfloat}
	\caption{Plot of  $f_{M}\left(\alpha,\lambda;x\right)$ and the generating function ${f\left(\alpha;x\right)}$, {for $\alpha=1.5,~\lambda=3$, and
			$\gamma_{3}=0.01$ varying $M=500,~1000$, $5000$}.}
	\label{fig: asymptotical behaviour}
\end{figure}
\begin{remark}\label{remark3}
	The function $f\left(\alpha;x\right)$ has a zero at the origin and is negative for $x\neq 0$, when $\gamma_{3}$ lies in the interval given in Lemma \ref{Lemma1}.
\end{remark}

Concerning the coefficient matrix $\mathcal{A}_{M}^{j+1}$ defined in equation~\eqref{eq:Coe_Matrix_form}, if $c_l(x,t)=c_r(x,t)=c$, then it is independent of $j$ since
\begin{equation}\label{eq:AMconst}
	\mathcal{A}_{M}^{j+1}=I_M-
	\frac{c\tau}{2h^{\alpha}}\left({B_{M,\lambda}^{\left(2,\alpha\right)}}+(B_{M,\lambda}^{\left(2,\alpha\right)})^{T}\right)
	=I_M-\frac{c\tau}{h^{\alpha}}\mathcal{\bf B}_{M,\lambda}^{\left(2,\alpha\right)},
\end{equation}
and hence the following result follows from the previous Proposition \ref{Proposition2}. From now on, we omit the superscript $j+1$.

\begin{corollary}\label{cor_symbol_coeff_matrix}
	Let us assume that $c_l(x,t)=c_r(x,t)=c$. Then, assured that $\frac{h^\alpha}{\tau}=o(1)$, the matrix-sequence $\{\frac{h^\alpha}{\tau}\mathcal{A}_{M}\}_{M\in\mathbb{N}}$
	is generated by the function 
	\begin{equation} \label{eq:falpha}
		f_{\alpha}(x)=-c \, f\left(\alpha;x\right).
	\end{equation} 
\end{corollary}

In the following we show that the generating function $f_{\alpha}(x)$ in \eqref{eq:falpha} gives the asymptotic spectral distribution in the case of constant diffusion coefficients. In other words $f_{\alpha}(x)$ is also the spectral symbol of the related matrix-sequence according to the definition below.

\begin{definition}\label{def5}
	Let $f:[a,b]\rightarrow\mathbb{C}$ be a measurable function. Let $\mathcal{C}_{0}(\mathbb{K})$  be the set of continuous functions with compact support over $\mathbb{K}\in\{\mathbb{C},\mathbb{R}^{+}_{0}\}$ and let $\{A_{M}\}_{M\in\mathbb{N}}$ be a sequence of matrices of size $M$ with eigenvalues $\mu_{j}(A_{M}),~j=1,2,\cdots,M$.
	We say that $\{A_{M}\}_{M\in\mathbb{N}}$ is distributed as the pair $\left(f,[a,b]\right)$ in the sense of the eigenvalues, and write
	\begin{align*}
		\{A_{M}\}_{M\in\mathbb{N}}\sim\left(f,[a,b]\right),
	\end{align*}
	if the following relation holds for all $F\in\mathcal{C}_{0}(\mathbb{C})$:
	\begin{align}\label{eq:Symbol_def}
		\lim_{M\rightarrow\infty}\frac{1}{M}\sum_{j=1}^{M}F\left(\mu_{j}(A_{M})\right)=\frac{1}{b-a}\int_{a}^bF(f(t))dt.
	\end{align}
	Thus, we write that $f$ is the (spectral) symbol of the matrix-sequence $\{A_{M}\}_{M\in\mathbb{N}}$.
\end{definition}

\begin{remark}\label{remark2}
	When $f$ is continuous, an informal interpretation of \eqref{eq:Symbol_def} is that when the matrix size is sufficiently large, the eigenvalues of $A_{M}$ can be approximated by a sampling of $f$ on a uniform equispaced grid of $[a,b]$.
\end{remark}

For Hermitian Toeplitz matrix-sequences, the following result due to Szeg\H{o} holds~\cite{grenander1984toeplitz}.

\begin{theorem}\label{theorem2}
	Let $f\in L^{1}([-\pi,\pi])$ be a real-valued function. Then,
	\begin{align*}
		\{T_{M}(f)\}_{M\in\mathbb{N}}\sim \left(f,[-\pi,\pi]\right),
	\end{align*}
	that is, the generating function of the sequence is also the (spectral) symbol. 
\end{theorem}

When the diffusion coefficients are constant and equal to $c$, according to Corollary \ref{cor_symbol_coeff_matrix}, the eigenvalue distribution of the coefficient matrix-sequence $\{\mathcal{A}_{M}\}_{M\in\mathbb{N}}$ properly scaled is
\[
\left\{\frac{h^\alpha}{\tau}\mathcal{A}_{M}\right\}\sim \left(f_{\alpha},[-\pi,\pi]\right).
\]

\begin{remark}
	For $\gamma_3=0$ the function $f\left(\alpha;x\right)$ coincides with the symbol of the non-tempered case retrieved in \cite{donatelli2020multigrid} and in \cite{lin2020accuracy} when considering the same shift $p_1=1$, $p_2=0.$
\end{remark}

%In our result we observed that $f_{M}\left(\alpha,\lambda;x\right)<0$ even 

The previous symbol analysis could be easily applied also to prove the stability of CN-TWSGD
in the case of constant diffusion coeffients, which is already proved in \cite{bu2021high} with other mathematical tools. Indeed,
in order to have the stability of the CN-TWSGD scheme \eqref{eq:Matrix_form}, the spectral radius of ${(I_{M}-\mathcal{M}_{M}^{j+1})^{-1}}{(I_{M}+\mathcal{M}_{M}^j)}$ should be less than one, which follows applying the following result.
\begin{lemma}\label{Lemma2}
	(Grenander-Szeg\H{o} theorem \cite{chan1991toeplitz}) Let $T_{n}$ be a Toeplitz matrix  with generating function $f$ belonging to the Wiener class. 
	%we denote by $\mu_{min}(T_{n})$ and $\mu_{max}(T_{n})$ the smallest and largest eigenvalues of $T_{n}$, respectively. If $f$ is a $2\pi$-periodic continuous real-valued function, defined on $\left[-\pi,\pi\right]$, then
	%	\begin{equation}\label{eq:GSTa}
	%		f_{min}\leq\mu_{min}(T_{n})\leq\mu_{max}(T_{n})\leq f_{max},
	%	\end{equation}
	%where $f_{min}$ and $f_{max}$ denote the minimum and maximum values of $f$. Moreover, 
	Denote with $f_{min}$ and $f_{max}$ the minimum and maximum values of $f$.
	If $f_{min}<f_{max}$, then all the eigenvalues of $T_{n}$ satisfy
	\begin{equation*}\label{eq:GSTb}
		f_{min}<\mu(T_{n})< f_{max},
	\end{equation*}
	for all $n>0$.%; and furthermore if $f_{min}\geq0$, then $T_{n}$ is positive definite.
\end{lemma}
It follows that, for $\lambda\geq0$ and $1\leq\alpha\leq 2$, when $\gamma_{3}$ lies in the interval given in Lemma \ref{Lemma1}, the symbol $f(\alpha; x)$ in \eqref{eq:symf} is negative, then the matrix $\mathcal{\bf B}_{M,\lambda}^{\left(2,\alpha\right)}$ is negative definite and hence the numerical scheme \eqref{eq:Matrix_form} is stable for  $c_l(x,t)=c_r(x,t)=c$ since all the eigenvalues of the matrix ${(I_{M}-\mathcal{M}_{M}^{j+1})^{-1}}{(I_{M}+\mathcal{M}_{M}^j)}$ are in modulus smaller than one.

From now on, since the generating function $f_\alpha$ is also the (spectral) symbol, we will use the two terms as synonyms. However, in a general setting, the two terms have a different meaning (see \cite{GLT-I} and references therein).

%%%%%%%%%%%%
\section{A smoothing analysis for multigrids applied to TFDEs}\label{sect:mgm}
Multigrid methods have already proven to be effective solvers as well as valid preconditioners for Krylov methods when numerically approaching $\mbox{FDEs}$ \cite{moghaderi2017spectral,donatelli2020multigrid,pang2012multigrid}. A multigrid method combines two iterative methods called smoother and coarse grid correction.  The coarser matrices can either be obtained by projection (Galerkin approach) or by rediscretization (geometric approach). In case only a single coarser level is given we talk about Two Grids Methods (TGMs), while in presence of more levels we talk about V-cycle. In this section, we investigate the convergence of TGM in Galerkin form for the discrete CN-TWSGD scheme when considering $c_{l}(x,t)=c_{r}(x,t)=c>0$. In this framework, the coefficient matrix $\mathcal{A}_M$ in \eqref{eq:AMconst} is Toeplitz so damped Jacobi as a smoother is a good choice \cite{sun1997note}.

Note that we are allowed to use MGMs as the coefficient matrix is symmetric positive definite, thanks to Lemma~\ref{Lemma2}, because its symbol $f_{\alpha}$ is nonnegative and not identically zero, cf. Corollary~\ref{cor_symbol_coeff_matrix}.

The TGM convergence analysis for TFDEs was already investigated in \cite{bu2021multigrid}. Here, we derive similar results using the symbol $f_{\alpha}$ in \eqref{eq:falpha} and the theory of multigrid methods for Toeplitz matrices, see \cite{fiorentino1991multigrid}.

Following the analysis in~\cite{chan1998multigrid}, given a symmetric positive definite matrix $A_M$, we call $D$ the diagonal of $A_M$. Moreover, as a matter of convenience, we only consider post-smoothing and call $S$ the post-smoothing iteration matrix, while with $P_M$ we denote a full rank prolongation matrix $P_M\in \mathbb{R}^{M\times k}$ with $k<M$. Then, the iteration matrix of Galerkin TGM is given by
\begin{equation*}\label{eq:TGM}
	TGM=S\Bigg[I_M-P_M(P_M^{T}{A}_{M}P_M)^{-1}P_M^{T}{A}_{M}\Bigg].
\end{equation*}

Thanks to the symmetric positive definite property of the matrix $A_M$, we can define the following inner products:
\begin{equation*}\label{eq:inner_Pro}
	\langle u_{1},u_{2}\rangle_{0}=\langle Du_{1},u_{2}\rangle,\quad \langle u_{1},u_{2}\rangle_{1}=\langle A_{M}u_{1},u_{2}\rangle,\quad \langle u_{1},u_{2}\rangle_{2}=\langle D^{-1}A_{M}u_{1},A_{M}u_{2}\rangle,
\end{equation*}
where $\langle \cdot,\cdot\rangle$ is the Euclidean inner product, and their respective norms $\vert\vert\cdot\vert\vert_{j}$, for $j=0,1,2$.  

\begin{theorem}[Ruge-St\"uben \cite{ruge1987algebraic}]\label{theorem4}
	Let $A_M$ be a symmetric positive definite matrix and $S$ be the post-smoothing iteration matrix. Assume that $\exists\ \sigma>0$ such that
	\begin{equation}\label{eq:Smoothing_pro}
		\vert\vert S{\rm{e}}_{h}\vert\vert_{1}^{2}\leq\vert\vert {\rm{e}}_{h}\vert\vert_{1}^{2}-\sigma\vert\vert {\rm{e}}_{h}\vert\vert_{2}^{2},\quad\forall~{\rm{e}}_{h}\in\mathbb{R}^{M},
	\end{equation}
	and that $\exists\ \delta>\sigma$ such that
	\begin{equation}\label{eq:App_Pro}
		\min_{y\in\mathbb{R}^{k}} \vert\vert {{\rm{e}}}_{h}-P_My\vert\vert_{0}^{2}\leq\delta\vert\vert {\rm{e}}_{h}\vert\vert_{1}^{2},\quad\forall~{\rm{e}}_{h}\in\mathbb{R}^{M},
	\end{equation}
	then 
	\begin{equation*}\label{eq:norm_TGM}
		\vert\vert \mbox{TGM}\vert\vert_{1}\leq\sqrt{1-\frac{\sigma}{\delta}}.
	\end{equation*}
\end{theorem}
The inequalities in equations \eqref{eq:Smoothing_pro} and \eqref{eq:App_Pro} are well known and go with the names of the smoothing property and the approximation property, respectively. To prove the TGM convergence, these two conditions are investigated separately in the next two subsections.

\subsection{Smoothing property}
It is well known that in case of a symmetric positive definite Toeplitz coefficient matrix, and Jacobi as smoother, the smoothing property \eqref{eq:Smoothing_pro} is satisfied whenever the smoother converges \cite{chan1998multigrid,ruge1987algebraic}.
For a symmetric positive definite matrix $A_M$, weighted Jacobi iteration matrix is $S_\omega=I_M-\omega D^{-1}A_M$ and its convergence is guaranteed if $ 0<\omega<\frac{2}{\rho(D^{-1}A_M)}$, with $D$ being the diagonal of $A_M$, and $\rho(D^{-1}A_M)$ the spectral radius of $D^{-1}A_M$.
In the constant coefficients case, i.e., $c_l(x,t)=c_r(x,t)=c$, the scaled coefficient matrix $\frac{h^\alpha}{\tau}\mathcal{A}_M$ is symmetric positive definite and has a Toeplitz structure such that 
$D=a_0I_M$, where $\tilde{a}_0$ is Fourier coefficient of order zero of $f_\alpha$ and 
$a_0=\tilde{a}_0 + \frac{h^\alpha}{\tau}$.
Therefore, according to Corollary \ref{cor_symbol_coeff_matrix}, we deduce that the weighted Jacobi satisfies the smoothing property \eqref{eq:Smoothing_pro} whenever
\begin{align*}
	0<\omega&<\xi=\frac{2a_{0}}{\vert\vert {f_\alpha}\vert\vert_{\infty}}.
\end{align*}
From the expression of $\mathcal{A}_M$ in \eqref {eq:AMconst}, we have
\begin{align*}
	a_{0}&=\frac{h^\alpha}{\tau}-c\left(g_{1,\lambda}^{(2,\alpha)}-\phi(\lambda)\right),\\
	&=\frac{h^\alpha}{\tau}-c(1-\frac{\alpha}{2}(\alpha+1)-\gamma_{3}(\alpha+2)),
\end{align*}
while, thanks to the monotonicity of $f_\alpha$ 
%(\cred{MARCO: credo che vada dimostrata}), 
we have
\[
\vert\vert {f_\alpha}\vert\vert_{\infty}=-f_\alpha(\pi)=c \, 2^{\alpha}(1-\alpha-4\gamma_{3}).
\]
Therefore, neglecting the term $\frac{h^\alpha}{\tau}$ in $a_0$, we have 
\begin{equation*}
	\xi \, \approx \, \frac{(1-\frac{\alpha}{2}(\alpha+1)-\gamma_{3}(\alpha+2))}
	{2^{\alpha-1}(1-\alpha-4\gamma_{3})}.
\end{equation*}

Figure \ref{fig:Xi} depicts $\xi$ varying $\alpha\in [1,2]$, and taking 11 equispaced values of $\gamma_3\in[0.001,0.06]$. For each value of $\gamma_3$ we get a curve. The solid part of each depicted curve corresponds to the values of $\alpha$ that are allowed with that choice of $\gamma_3$. The values $\gamma_3=0.01,0.00235$, represented by a different marker, are those we use in our numerical tests. We note that $\xi$ is always larger than $0.7$ and in particular of the value $0.5$ proved in \cite{bu2021multigrid}.

\begin{figure}
	\centering
	\includegraphics[width=0.6\textwidth]{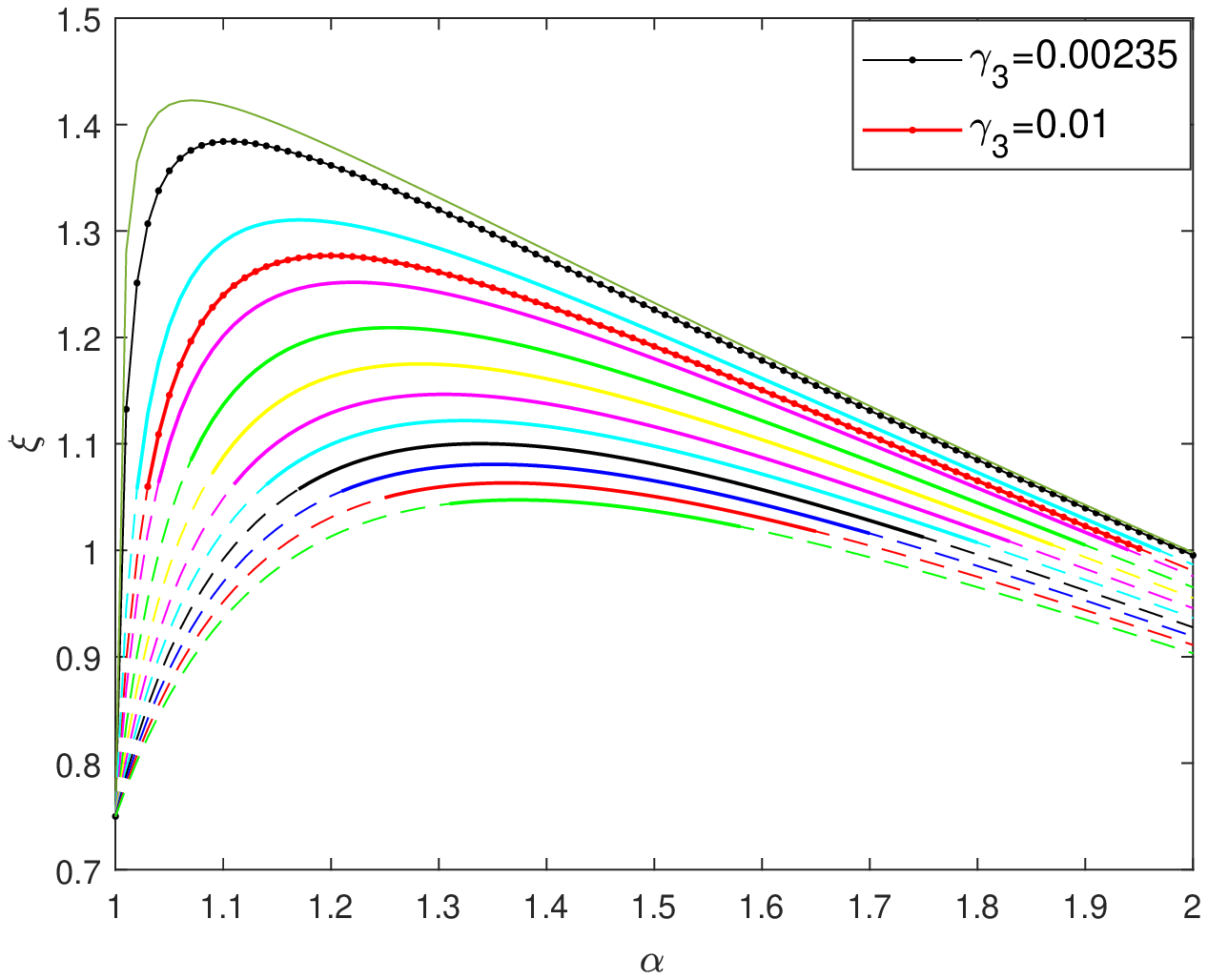}
	\caption{Plots of $\xi$ for different values of $\gamma_{3}$ ($\gamma_{3}>0$, $1\leq\alpha\leq2$).}
	\label{fig:Xi}
\end{figure}

Following the idea behind the optimum parameter for the Laplacian operator which is $2/3$ for the range $[0,1]$, in order to choose a good relaxation parameter $\omega \in [0,\xi]$ we propose 
\begin{equation}\label{eq:omega}
	\omega^\star=\frac{2}{3}\xi,
\end{equation}
which provides a good convergence rate as confirmed in the numerical results in Section \ref{sect:num}. 	

\subsection{Approximation property}

The approximation property proved in \cite{bu2021multigrid} can be easily deduced combining the symbol $f_\alpha$ with the convergence analysis in \cite{capizzano2002NM}. 	

Let the projector $P_M$ be the classical linear interpolation such that
\[
P_M^T=\frac{1}{2}
\begin{bmatrix}
	1 & 2 & 1 & & & & &\\
	& & 1 & 2 & 1 & & & \\
	& & &  \ddots & \ddots & \ddots & & &\\
	& & & &  & 1 & 2 & 1 
\end{bmatrix}
=K_MT_M(p),
\]
where $K_M\in\mathbb{R}^{\frac{M-1}{2} \times M}$ for $M$ odd, i.e., $[K_M]_{i,j}=1$ for $j=2i$, $i=1\dots,\frac{M-1}{2}$, and $p(x)=1+\cos(x)$.

Thanks to Corollary \ref{cor_symbol_coeff_matrix} and Remark \ref{remark3}, the symbol $p$ of the projector satifies
\begin{equation} \label{eq:tgm}
	\lim_{x\to 0} \sup \frac{p(x+\pi)^2}{f_\alpha(x)}=0.
\end{equation}
Hence, thanks to Lemma 5.2 in \cite{capizzano2002NM}, the TGM has a linear convergence. 

Note that the limit \eqref{eq:tgm} vanishes even removing the power two at the numerator.
This gives a linear convergence even replacing the TGM with the V-cycle, see \cite{arico2004v}. Moreover, the proposed multigrid method is so robust that the Galerkin approach used in the theoretical analysis can be replaced with the geometric approach. Indeed, the Galerkin approach is too expensive when applied to a full-matrix because the algebraic structure at the coarser levels is lost, which is crucial to remain within a $O(M\log(M))$ computational cost for the matrix-vector product. Moreover, the rediscretization matrix-sequence and the matrix-sequence obtained by Galerkin projections are spectrally equivalent and this represents a motivation for the good convergence speed of the method, when using the rediscretization as well.  

Finally, even in the case of variable diffusion coefficients, the proposed multigrid method has a linear convergence assuming that the coefficient functions are strictly positive and bounded, see \cite{capizzano2002NM}.

\section{\bf Numerical Examples}\label{sect:num}
In this section, we present some numerical examples, taken from \cite{bu2021multigrid,deng2018variational}, to verify the effectiveness of the MGMs introduced in the previous section. In order to improve the robustness of MGMs it is common to use them as preconditioners for Krylov methods.
%Aiming to increase the robustness of the Krylov methods, different type of preconditioners have been used in literature. 
In our case, we  apply multigrid preconditioner, the Chan circulant preconditioner $P_C$ \cite{chan1993fft}, and the Laplacian preconditioner $P_{2}$, which in \cite{donatelli2016spectral} was shown to be efficient for $\alpha$ not far from $2$. Our multigrid solver consists in a V-cycle with $\nu_1$ and $\nu_2$ iterations of pre and post-smoother, respectively, which we shorten by employing the notation $V(\nu_1,\nu_2)$. In case where $V(\nu_1,\nu_2)$ is used as preconditioner, we denote it by $PV(\nu_1,\nu_2)$. %Finally, we denote by $\omega^\star$ the Jacobi relaxation parameter proposed in \eqref{eq:omega}. 

In the following tables, we fix $M=N$, where $M$ and $N$ denote the number of spatial and time grid points. The preconditioned CG and GMRES are computationally performed using built-in \textit{pcg} and \textit{gmres} Matlab functions, respectively. The stopping criterion is $\frac{\| r^{k}\|}{\| r^{0}\|}<\mbox{tol}$, where $\|\cdot\|$ denotes the Euclidean norm, $r^{k}$ is the residual vector at the $k$-th iteration and $\mbox{tol}=10^{-7}$ is the tolerance. The initial guess is fixed as the zero vector. In case of a time-dependent TFDE, the reported iterations are the average number of iterations at each time-step and the initial guess in the solution computed at the previous time step.

{\bf Example 1.} Consider the TFDE
\begin{align}\label{eq_Ex1}
	\begin{cases}
		\frac{\partial u(\,x,t)\,}{\partial t}=\,{_{0}{\mathbb{{D}}_{x}^{\alpha,\lambda}}u(\,x,t)\,}+\,{_{x}{\mathbb{{D}}_{1}^{\alpha,\lambda}}u(\,x,t)\,}+f(\,x,t)\,,\\%\quad (\,x,t)\,\in(\,0,1)\,\times(\,0,1)\,\\
		u(\,0,t)\ =0,\quad u(\,1,t)\,=0,\quad t\in[\,0,1]\,,\\
		u(\,x,0)\,=x^3\left(1-x\right)^3,\quad x\in[\,0,1]\,,
	\end{cases}       
\end{align}
with source term and exact solution taken from \cite{bu2021multigrid},
\begin{align*}
	f\left(x,t\right)=&\left(2\lambda^{\alpha}-1\right){\rm{e}}^{-t} x^3\left(1-x\right)^3-{\rm{e}}^{-\lambda x-t}\big(_{0}{\mathcal{D}_{x}^{\alpha}}\big[{\rm{e}}^{\lambda x}\left(x^3-3x^4+3x^5-x^6\right)\big]\big)\newline\\
	&-{\rm{e}}^{\lambda x-t}\big(_{x}{\mathcal{D}_{1}^{\alpha}}\big[{\rm{e}}^{-\lambda x}\left((1-x)^3-3(1-x)^4+3(1-x)^5-(1-x)^6\right)\big]\big),
	\\u(\,x,t)\,=&{\rm{e}}^{-t} x^3\left(1-x\right)^3.\newline\\
\end{align*}
%Here we compute $f(\,x,t)\,$ by using following formulae
%\begin{align*}
%	_{0}{\mathcal{D}_{x}^{\alpha}}\left({\rm{e}}^{\lambda x}{x^m}\right)=&{_{0}{\mathcal{D}_{x}^{\alpha}}}\Big[\sum_{n=0}^{\infty}{\frac{\lambda^{n}}{n!}x^{n+m}}\big]=\big[\sum_{n=0}^{\infty}{\frac{\lambda^{n}\Gamma\left(n+m+1\right)}{n!\Gamma\left(n+m-\alpha+1\right)}x^{n+m-\alpha}}\Big],\newline\\
%	_{x}{\mathcal{D}_{1}^{\alpha}}\left({\rm{e}}^{-\lambda x}{(1-x)^m}\right)=&{\rm{e}}^{-\lambda}{_{x}{\mathcal{D}_{1}^{\alpha}}}\Big[\sum_{n=0}^{\infty}{\frac{\lambda^{n}}{n!}(1-x)^{n+m}}\Big]={\rm{e}}^{-\lambda}\Big[\sum_{n=0}^{\infty}{\frac{\lambda^{n}\Gamma\left(n+m+1\right)}{n!\Gamma\left(n+m-\alpha+1\right)}(1-x)^{n+m-\alpha}}\Big].
%\end{align*}
%%%%%%%%%%%
We recall that in \cite{bu2021multigrid} the convergence of MGMs was proven for $\alpha\in(1.26,1.71)$ and $\omega\in(0,0.5]$. In their first example the authors considered the TFDE in equation \eqref{eq_Ex1} with $\lambda=0.5,\alpha=1.5$ and $\gamma_3=0.01$ to provide numerical evidences that support their theoretical results. Here we consider similar settings by fixing $\gamma_3=0.01$, $\lambda\in\lbrace 0;2;10\rbrace$, $\alpha\in\lbrace 1.2;1.5;1.8\rbrace$ and test the relaxation parameter $\omega^\star$ in \eqref{eq:omega}.

Table \ref{tab:T11} shows the iterations to tolerance of $V(1,1)$ as standalone solver for the above combinations of parameters and different values of $\omega$. First we observe that varying $\lambda$ does not seem to significantly affect the overall iterations, which somehow reflect the fact that the generating function for the tempered case is the same as in the non tempered one, i.e. $\lambda=0$. Moreover, it is not surprising that $\lambda=0.5$, as taken in \cite{bu2021multigrid}, yields similar results as for $\lambda=0$ (results not reported). We note that for the choice of the parameters $\alpha=1.2$ and $\alpha=1.8$, $\omega^\star$ leads to the smallest number of iterations, while for $\alpha=1.5$ our optimal weight $\omega^\star$ seems to slightly overestimate the numerically optimal one, since lower values of $\omega$ lead to fewer iterations. In this regard, we recall that $\omega^\star$ is directly linked to the symbol, which in turn is obtained by letting the size of the Toeplitz tend to infinity. This means that by increasing the matrix size, $\omega^\star$ becomes more suitable and, therefore, the iterations should decrease, which seems to be our case.

\begin{table}
	\caption{Example 1 - average number of iterations to tolerance of $V(1,1)$ with $\omega^\star=0.85,0.79,0.71$ for $\alpha=1.2,1.5,1.8$, respectively, and ${\gamma_{3}=0.01}$.}
	\begin{small}
		\setlength{\tabcolsep}{8pt}
		\begin{center}
			\begin{tabular}{c c c c c c c c c c c}
				\hline
				\hline
				\noalign{\vskip 1mm}
				%\multirow{1}{*} {}&
				%\multicolumn{4}{c}{${}$} &
				%\multicolumn{2}{c}{$V(1,1)$} &
				%\multicolumn{1}{c}{${}$} \\
				%\cmidrule(r){7-9}  
				%\cmidrule(r){4-9}%   \cmidrule(r){14-16}
				$\lambda$&$\alpha$&$M=N$& $\omega=0.5$& $\omega=0.6$& $\omega=0.7$& $\omega=0.8$&$\omega=0.9$&$\omega^\star$\\ 
				\hline
				$     $ &${}$&$2^6$&${7}$&${6}$&${5}$&${4}$&${4}$&${4}$\\
				$       $ &$1.2$&$2^7$&${6}$&${5}$&${4}$&${4}$&${4}$&${3}$\\
				$     $ &${}$&$2^8$&${6}$&${5}$&${4}$&${3}$&${4}$&${3}$\\
				$     $ &${}$&$2^9$&${5}$&${4}$&${4}$&${3}$&${3}$&${3}$\\
				\cline{2-9}
				$     $ &${}$&$2^6$&${8}$&${6}$&${5}$&${6}$&${9}$&${6}$\\
				$\boldsymbol{0}$
				&$1.5$&$2^7$&${7}$&${6}$&${5}$&${6}$&${9}$&${6}$\\
				$     $ &${}$&$2^8$&${7}$&${6}$&${5}$&${6}$&${8}$&${6}$\\
				$     $ &${}$&$2^9$&${6}$&${6}$&${5}$&${5}$&${8}$&${5}$\\
				\cline{2-9}
				$     $ &${}$&$2^6$&${11}$&${8}$&${8}$&${11}$&${19}$&${8}$\\
				$       $ &$1.8$&$2^7$&${10}$&${8}$&${7}$&${11}$&${18}$&${7}$\\
				$     $ &${}$&$2^8$&${10}$&${8}$&${7}$&${10}$&${18}$&${7}$\\
				$     $ &${}$&$2^9$&${9}$&${7}$&${7}$&${10}$&${17}$&${7}$\\
				\hline
				\hline
				$     $ &$~$&$2^6$&${7}$&${5}$&${5}$&${4}$&${4}$&${4}$\\
				$      $ &$1.2$&$2^7$&${6}$&${5}$&${4}$&${4}$&${4}$&${3}$\\
				$     $ &$~$&$2^8$&${6}$&${5}$&${4}$&${3}$&${4}$&${3}$\\
				$     $ &$~$&$2^9$&${5}$&${4}$&${4}$&${3}$&${3}$&${3}$\\
				\cline{2-9}
				$     $ &$~$&$2^6$&${8}$&${7}$&${5}$&${6}$&${9}$&${6}$\\
				$ \boldsymbol{2} $ &$1.5$&$2^7$&${7}$&${6}$&${5}$&${6}$&${9}$&${6}$\\
				$    $ &${}$&$2^8$&${6}$&${5}$&${5}$&${6}$&${8}$&${6}$\\
				$    $ &${}$&$2^9$&${6}$&${5}$&${5}$&${5}$&${8}$&${5}$\\
				\cline{2-9}
				$     $ &${}$&$2^6$&${11}$&${9}$&${8}$&${11}$&${19}$&${8}$\\
				$       $ &$1.8$&$2^7$&${10}$&${8}$&${7}$&${11}$&${19}$&${7}$\\
				$     $ &${}$&$2^8$&${10}$&${8}$&${7}$&${10}$&${18}$&${7}$\\
				$     $ &${}$&$2^9$&${9}$&${7}$&${7}$&${10}$&${17}$&${7}$\\
				\hline
				\hline
				$     $ &${}$&$2^6$&${7}$&${5}$&${4}$&${4}$&${5}$&${5}$\\
				$       $ &$1.2$&$2^7$&${6}$&${5}$&${4}$&${3}$&${4}$&${4}$\\
				$     $ &${}$&$2^8$&${5}$&${4}$&${4}$&${3}$&${4}$&${3}$\\
				$     $ &${}$&$2^9$&${5}$&${4}$&${3}$&${3}$&${3}$&${3}$\\
				\cline{2-9}
				$     $ &${}$&$2^6$&${9}$&${7}$&${6}$&${8}$&${11}$&${7}$\\
				$\boldsymbol{10}$&$1.5$&$2^7$&${8}$&${6}$&${5}$&${7}$&${10}$&${7}$\\
				$     $ &${}$&$2^8$&${7}$&${5}$&${5}$&${6}$&${9}$&${6}$\\
				$     $ &${}$&$2^9$&${6}$&${5}$&${4}$&${6}$&${8}$&${6}$\\
				\cline{2-9}
				$     $ &${}$&$2^6$&${11}$&${9}$&${8}$&${12}$&${21}$&${8}$\\
				$     $ &$1.8$&$2^7$&${10}$&${8}$&${7}$&${11}$&${20}$&${8}$\\
				$     $ &${}$&$2^8$&${10}$&${8}$&${7}$&${11}$&${18}$&${7}$\\
				$     $ &${}$&$2^9$&${9 }$&${7}$&${7}$&${10}$&${17}$&${7}$\\
				\hline
			\end{tabular}
		\end{center}
	\end{small}
	\label{tab:T11}
\end{table}
%%%%%%%%
%%%%%%%%%%%%%%%%%%%%%%%%%%%%%%%%%%%%%%%%%%%%%%%%%%%%%%%%%%%%%%%%%%%%%%%%%%%%%%%%%%%%%%%%%%%%%%%%%%%%%%%%%%%%%%
{\bf Example 2.} In this example, taken from \cite{deng2018variational}, we consider the steady tempered fractional advection-dispersion model defined as
\begin{equation*}\label{eq_Ex2}
	-\frac{1}{2}\big({_{a}{\mathbb{{D}}_{x}^{\alpha,\lambda}}u(\,x)\,}+\,{_{x}{\mathbb{{D}}_{b}^{\alpha,\lambda}}u(\,x)\,}\big)=f(\,x)\,,
\end{equation*}
with boundary conditions $u(0)=u(1)=0$, where the source term $f\left(x\right)$ is extracted from the exact solution $u(x)=\left(1-x\right)^3-{\rm{e}}^{3x}\left(1-x\right)$.

Table \ref{tab:TI0.5} shows the iteration count to tolerance when considering the methods $V(0,1)$, $V(1,1)$, $PV(1,1)$ and some comparison solvers like the unpreconditioned CG, the Laplacian preconditioner $P_2$ and the circulant preconditioner $P_\text{C}$.

By removing the time dependency, the ill-conditioning of the coefficient matrix increases, since the identity matrix in equation \eqref{eq:Coe_Matrix_form} disappears. Indeed, in Table \ref{tab:TI0.5}, fixing $\lambda=3$, we note an increase of the overall iterations to tolerance with respect to Table \ref{tab:T11}. Nevertheless, for any tested weight $\omega$, the iterations of the tested V-cycles are stable with respect to the increasing size of the linear system, while the iterations of preconditioners $P_C$ and $P_2$ tend to increase. We note that, when $\alpha\approx 2$, the preconditioner $P_2$ yields a low amount of iterations and shows linear convergence with respect to the size, which is in accordance with the results in the non-tempered case \cite{donatelli2016spectral}.

Choosing the weight $\omega^\star$ we obtain the minimal iteration count of $V(1,1)$. Removing the pre-smoothing iteration we observe an increase of $\sim5$ iterations, but the solver is still robust enough to yield a linear convergence with respect to the matrix size. More robustness can be reached by using CG as main solver and one iteration of $V(1,1)$. In this case, the iteration in Table \ref{tab:TI0.5} are shown to reduce almost by a half. Note that the iteration matrix of $V(0,1)$ is not symmetric and hence it cannot by applied as preconditioner for CG.

%%%%%%%%%%%%%%%%%%%%%%%%%%%%%%%%%%%%%%%%%%%%%%%%%%%%%%%
\begin{table}
	\caption{Example 2 with $\lambda=3$ - number of iterations to tolerance of V-cycle and preconditioned CG with $\omega^\star=0.85,0.75,0.69$ for $\alpha=1.4,1.7,1.9$, respectively, and $\gamma_3=0.00235$ .}
	\begin{small}
		\setlength{\tabcolsep}{2pt}
		\begin{center}
			\begin{tabular}{c c c c c c c c c c c c c c}
				\hline
				\hline
				\noalign{\vskip 1mm}
				\multirow{1}{*} {}&
				\multicolumn{2}{c}{${}$} &
				%	\multicolumn{2}{c}{${}$} &
				\multicolumn{6}{c}{${V(1,1)}$} &
				\multicolumn{1}{c}{${V(0,1)}$} &
				%				\multicolumn{2}{c}{${PV(0,1)}$} &
				\multicolumn{1}{c}{${PV(1,1)}$} &
				\multirow{2}{*}{$P_\text{C}$} &
				\multirow{2}{*}{$P_2$}&
				\multicolumn{1}{c}{}\\
				\cmidrule(r){4-9}\cmidrule(r){10-10}\cmidrule(r){11-11}
				${\alpha}$& ${\mbox{M=N}}$&${\mbox{CG}}$ & $\;{\omega=\mbox{0.5}}\;$ & {$\omega=\mbox{0.6}$} & {$\;\omega=\mbox{0.7}\;$}& {$\;\omega=\mbox{0.8}\;$}& {$\;\omega=\mbox{0.9}\;$}&$\;\omega^\star\;$&$\omega^\star$&$\omega^\star$ \\ \hline
				%%%%%%%%%%%%%%%%%%%%%%%%%%%%%%%
				$           $&$2^7$   &${75 }$&$13$&$11$&$9 $&$9$&${11}$&${9}$&${15}$&${6}$&${12}$&${20}$\\
				$           $&$2^8$   &${121}$&$14$&$11$&$10$&$9$&${10}$&${9}$&${14}$&${6}$&${14}$&${24}$\\
				$        1.4$&$2^9$  &${198}$&$14$&$12$&$10$&$9$&${10}$&${9}$&${13}$&${6}$&${16}$&${29}$\\
				$       $&$2^{10}$
				&${322}$&$14$&$12$&$10$&$9$&${10}$&${9}$&${13}$&${6}$&${17}$&${35}$\\
				\hline
				$           $&$2^7$   &${102}$&$14$&$11$&$10$&$11$&${17}$&$10$&$15$&${6}$&${15}$&${13}$\\
				$           $&$2^8$   &${185}$&$14$&$12$&$10$&$11$&${16}$&$10$&$15$&${6}$&${17}$&${15}$\\
				$        1.7$&$2^9$   &${334}$&$15$&$12$&$10$&$11$&${16}$&$10$&$14$&${6}$&${21}$&${17}$\\
				$           $&$2^{10}$&${603}$&$15$&$12$&$10$&$11$&${15}$&$10$&$14$&${6}$&${24}$&${19}$\\
				\hline
				$           $&$2^7$   &${123}$&$14$&$11$&$10$&${15}$&${27}$&$10$&$15$&${7}$&${17}$&${8}$\\
				$           $&$2^8$   &${238}$&$14$&$12$&$11$&${15}$&${27}$&$11$&$15$&${7}$&${21}$&${9}$\\
				$        1.9$&$2^9$   &${461}$&$15$&$12$&$11$&${14}$&${27}$&$11$&$15$&${7}$&${26}$&${10}$\\
				$           $&$2^{10}$&${891}$&$15$&$12$&$11$&${14}$&${26}$&$11$&$15$&${7}$&${32}$&${10}$\\
				\hline
			\end{tabular}
		\end{center}
	\end{small}
	\label{tab:TI0.5}
\end{table}

%%%%%%%%%%%%%%%%%%%%%%%%%%%%%%%%%%%%%%%%%%%%%%%%%%%%%%%%%%%%%%%%%%%%%%%%%%%%%
{\bf Example 3.} Let us consider the TFDE
\begin{equation*}\label{eq_Ex5}
	-\frac{3}{10}{_{a}{\mathbb{{D}}_{x}^{\alpha,\lambda}}u(\,x)\,}-\frac{7}{10}\,{_{x}{\mathbb{{D}}_{b}^{\alpha,\lambda}}u(\,x)\,}=f(\,x)\,,
\end{equation*}
where the source term $f\left(x\right)$ is extracted from the exact solution $u(\,x)\,=\left(1-x\right)^3-{\rm{e}}^{3 x}\left(1-x\right)$.

In this case the resulting linear system is non-symmetric, due to the different weighting of the two tempered fractional operators. Nevertheless, Table \ref{tab:TI0.7} shows multigrid to be robust enough to deal with the asymmetry, and the iterations also show a slight reduction with respect to the results in Table \ref{tab:TI0.5}.

%\cred{non capisco come possa funzionare meglio V-cycle in questo caso con GMRES e matrice non simmetrica invece dell Example 2 con CG e matrice simmetrica.}

%%%%%%%%%%%
\begin{table}
	\caption{Example 3 with $\lambda=3$ - number of iterations to tolerance of V-cycle and preconditioned GMRES with $\omega^\star=0.85,0.75,0.69$ for $\alpha=1.4,1.7,1.9$, respectively, and $\gamma_3=0.00235$.}
	\begin{small}
		\setlength{\tabcolsep}{2.5pt}
		\begin{center}
			\begin{tabular}{c c c c c c c c c c c c c c c}
				\hline
				\hline
				\noalign{\vskip 1mm}
				%\multirow{1}{*} {}&
				%\multicolumn{2}{c}{${}$} &
				%\multicolumn{1}{c}{${V(1,1)}$} &
				%\multicolumn{1}{c}{${V(0,1)}$} &
				%\multicolumn{1}{c}{${PV_{\alpha}(0,1)}$}&
				%\multicolumn{1}{c}{${PV_{\alpha}(1,1)}$}&
				%\multicolumn{1}{c}{${}$} \\
				%\cmidrule(r){4-4}\cmidrule(r){5-5}\cmidrule(r){6-6}\cmidrule(r){7-7} 
				${\alpha}$& ${\mbox{M=N}}$&${\mbox{GMRES}}$&$V(1,1)$&${V(0,1)}$& ${PV(1,1)}$& ${PV(0,1)}$ & ${P_C}$&$P_{2}$\\ \hline
				
				$           $&$2^7$   &${101}$&${9}$&${15}$&${5}$&${7}$&${11}$&${18}$\\
				$           $&$2^8$   &${180}$&${9}$&${14}$&${5}$&${6}$&${12}$&${20}$\\
				$        1.4$&$2^9$   &${320}$&${9}$&${13}$&${4}$&${6}$&${14}$&${21}$\\
				$           $&$2^{10}$&${567}$&${9}$&${13}$&${4}$&${6}$&${16}$&${22}$\\
				\hline
				$           $&$2^7$   &${115}$&${9}$&${15}$&${4}$&${7}$&${13}$&${11}$\\
				$           $&$2^8$   &${217}$&${9}$&${15}$&${4}$&${6}$&${15}$&${11}$\\
				$        1.7$&$2^9$   &${406}$&${9}$&${14}$&${4}$&${6}$&${18}$&${12}$\\
				$           $&$2^{10}$&${761}$&${9}$&${14}$&${4}$&${6}$&${22}$&${12}$\\
				\hline
				$           $&$2^7$   &${125}$&${10}$&${15}$&${5}$&${7}$&${15}$&${6}$\\
				$           $&$2^8$   &${246}$&${10}$&${15}$&${4}$&${7}$&${18}$&${6}$\\
				$        1.9$&$2^9$   &${482}$&${10}$&${15}$&${4}$&${7}$&${21}$&${6}$\\
				$           $&$2^{10}$&${944}$&${11}$&${15}$&${4}$&${7}$&${27}$&${6}$\\
				\hline
			\end{tabular}
		\end{center}
	\end{small}
	\label{tab:TI0.7}
\end{table}
%%%%%%%%%%%-2D-%%%%%%%%

\section{2D Problems}\label{sect:2D}
Here we consider the two dimensional TFDE given by

\begin{equation}\label{eq:Ex6}
	\begin{cases}
		\frac{\partial u(x,y,t)}{\partial t}={C}_{l}(x,y){_{a_{1}}{\mathbb{{D}}_{x}^{\alpha,\lambda_{1}}}u(x,y,t)}+{C}_{r}(x,y){_{x}{\mathbb{{D}}_{b_{1}}^{\alpha,\lambda_{1}}}u(x,y,t)}+{E}_{l}(x,y){_{a_{2}}{\mathbb{{D}}_{y}^{\beta,\lambda_{2}}}u(x,y,t)}\\
		\quad\quad\quad\quad+\mbox{E}_{r}(x,y){_{y}{\mathbb{{D}}_{b_{2}}^{\beta,\lambda_{2}}}u(x,y,t)}+f(x,y,t),\quad (x,y,t)\in\Omega\times(0,T],\\
		u(x,y,t)=0,\quad~~\quad (x,y,t)\in\partial\Omega\times[0,T],\\
		u(x,y,0)=u_{0},\quad\quad (x,y)\in\Omega\times[0,T],
	\end{cases}
\end{equation}
where ${C}_{s},{E}_{s},\ s\in\lbrace l,r\rbrace$ are the non-negative diffusion coefficients and $\Omega=[a_1,b_1]\times[a_2,b_2]$ and $T$ are the spatial domain and the final time-step, respectively.

In order to introduce the CN-WSGD scheme, let us fix $M_1,M_2,N\in\mathbb{N}$ and discretize the domain $\Omega\times[0,T]$ with
\begin{equation*}\label{3D_GRID}
	\begin{array}{lllll}
		x_j=a_1+jh_x,		&& h_x=\frac{b_1-a_1}{M_1+1}, 		&& j=0,...,M_1+1,  	\\
		y_j=a_2+jh_y,	    && h_y=\frac{b_2-a_2}{M_2+1},  		&& j=0,...,M_2+1,  	\\
		t^j=j\tau, 		    && \tau=\frac{T}{N}, 				&& j=0,...,N,	 	\\
		t^{j+\frac{1}{2}}=\frac{t^{j}+t^{j+1}}{2}, && 		    && j=1,...,N.
	\end{array}
\end{equation*}
For $s\in\lbrace l,r\rbrace$, let us now introduce the following $M$-dimensional vectors, with $M=M_{1}M_{2}$
\begin{align*}
	{\bf{u}}^{j}=&[\,u_{1,1}^{j},u_{2,1}^{j},\cdots,u_{M_{1},1}^{j},u_{1,2}^{j},u_{2,2}^{j},\cdots,u_{M_{1},2}^{j},\cdots,u_{1,M_{2}}^{j},u_{2,M_{2}}^{j},\cdots,u_{M_{1},M_{2}}^{j}],\\
	{\bf{C}}_s=&[\,{C}^{s}_{1,1},{C}^{s}_{2,1},\cdots,{C}^{s}_{M_{1},1},{C}^{s}_{1,2},{C}^{s}_{2,2},\cdots,{C}^{s}_{M_{1},2},\cdots,{C}^{s}_{1,M_{2}},{C}^{s}_{2,M_{2}},\cdots,{C}^{s}_{M_{1},M_{2}}],\\
	{\bf{E}}_{s}=&[\,{E}^{s}_{1,1},{E}^{s}_{2,1},\cdots,{E}^{s}_{M_{1},1},{E}^{s}_{1,2},{E}^{s}_{2,2},\cdots,{E}^{s}_{M_{1},2},\cdots,{E}^{s}_{1,M_{2}},{E}^{s}_{2,M_{2}},\cdots,{E}^{s}_{M_{1},M_{2}}],\\
	{\bf{F}}^{j+1}=&[{f}^{j+\frac{1}{2}}_{1,1},{f}^{j+\frac{1}{2}}_{2,1},\cdots,{f}^{j+\frac{1}{2}}_{M_{1},1},{f}^{j+\frac{1}{2}}_{1,2},{f}^{j+\frac{1}{2}}_{2,2},\cdots,{f}^{j+\frac{1}{2}}_{M_{1},2},\cdots,{f}^{j+\frac{1}{2}}_{1,M_{2}},{f}^{j+\frac{1}{2}}_{2,M_{2}},\cdots,{f}^{j+\frac{1}{2}}_{M_{1},M_{2}}],
\end{align*}
where $u_{i_1,i_2}^j=u(x_{i_1},y_{i_2})$, $C_{i_1,i_2}^s=C(x_{i_1},y_{i_2})$, $E_{i_1,i_2}^s=E(x_{i_1},y_{i_2})$, and $f_{i_1,i_2}^{j+\frac{1}{2}}=j(x_{i_1},y_{i_2},t^{j+\frac{1}{2}})$.

The 2D space CN-TWSGD scheme is obtained through the same procedure of discretization as in Section \ref{subsec:CN-TWSGD}, i.e., combining the CN discretization in time with the TWSGD formula for both tempered fractional derivatives with respect to $x$ and $y$. This leads to the time stepping scheme
\begin{equation*}\label{eq:Matrix_form_2D}
	\left(I_{M}-\mathcal{M}_{x,M}-\mathcal{M}_{y,M}\right)\mathbf{u}^{j+1}=\left(I_M+\mathcal{M}_{x,M}+\mathcal{M}_{y,M}\right)\mathbf{u}^{j}+\tau{\mathbf{F}^{j+1}},
\end{equation*}
where% the matrices $\mathcal{M}_{x,M}$ and $\mathcal{M}_{y,M}$ are
\begin{align*}
	\mathcal{M}_{x,M}&=r_{1}\left[\mbox{C}_{l}\left(I_{M_{2}}\otimes B^{(2,\alpha)}_{M_{1},\lambda_{1}}\right)+\mbox{C}_{r}\left(I_{M_{2}}\otimes (B^{(2,\alpha)}_{M_{1},\lambda_{1}})^{T}\right)\right]-s_{1}\left(\mbox{C}_{l}-\mbox{C}_{r}\right)\left(I_{M_{2}}\otimes H_{2,M_{1}}\right),\\
	\mathcal{M}_{y,M}&=r_{2}\left[\mbox{E}_{l}\left(B^{(2,\beta)}_{M_{2},\lambda_{2}}\otimes I_{M_{1}}\right)+\mbox{E}_{r}\left((B^{(2,\beta)}_{M_{2},\lambda_{2}})^{T}\otimes I_{M_{1}} \right)\right]-s_{2}\left(\mbox{E}_{l}-\mbox{E}_{r}\right)\left(H_{2,M_{2}}\otimes I_{M_{1}}\right),
\end{align*}
with 
\begin{equation*}
	{C}_{s}=\mbox{diag}({\bf{C}}_{s}),\quad {E}_{s}=\mbox{diag}({\bf{E}}_{s}),\ s\in\lbrace l,r\rbrace
\end{equation*}
and the scaling factors $r_{1}=\frac{\tau}{2h^{\alpha}_{x}}$, $r_{2}=\frac{\tau}{2h^{\beta}_{y}}$,  $s_{1}=\frac{\alpha\tau\lambda_{1}^{\alpha-1}}{4h_{x}}$ and $s_{2}=\frac{\beta\tau\lambda_{2}^{\beta-1}}{4h_{y}}$. 

Summing up, the time-stepping scheme is
\begin{equation*}\label{eq:Final_Lin_sys_2D}
	{\mathcal{A}_{(\alpha,\beta),M}}\mathbf{u}^{j+1}={b^{j+1}}.
\end{equation*}
with
\begin{align*}
	{\mathcal{A}_{(\alpha,\beta),M}}&=\left(I_{M}-\mathcal{M}_{x,M}-\mathcal{M}_{y,M}\right),\nonumber\\
	{b^{j+1}}&=\left(I_M+\mathcal{M}_{x,M}+\mathcal{M}_{y,M}\right)\mathbf{u}^{j}+\tau{\mathbf{F}^{j+1}}.
\end{align*}

We note that, similarly to the non-tempered case in \cite{moghaderi2017spectral} with constant and equal diffusion coefficients, from Proposition \ref{Proposition2} and the properties of the Kronecker product, 
${\mathcal{A}_{(\alpha,\beta),M}}$ is a symmetric positive definite Block-Toeplitz with Toeplitz Blocks (\mbox{BTTB}) matrix, whose symbol follows directly from the 1D case. Let $h_x^\alpha=h_y^\beta,\tau\rightarrow 0$, then $r_1=r_2$ and
\begin{equation*}
	\left\{\frac{1}{r_1}\mathcal{A}_{(\alpha,\beta),M}\right\}_{M\in\mathbb{N}}\sim \mathcal{F}(x,y)=-c \, f\left(\alpha;x\right) -e \, f\left(\beta;y\right). 
\end{equation*}
From a numerical point of view, MGMs work even if the constraint $h_x^\alpha=h_y^\beta$ is not satisfied, as long as there is not too much anisotropy, i.e., $\alpha$ is not that far from $\beta$.

Since the optimum Jacobi parameter for the 2D Laplacian is $\omega=\frac{4}{5}$ (see \cite{Trot}), according to the 1D case, the relaxation parameter of Jacobi $\omega^\star$ is computed as 
\begin{equation}\label{eq:w2D}
	\omega^\star=\frac{4}{5}\xi,    
\end{equation}
where $\xi=\frac{2\hat{\mathcal{F}}_0}{\|\mathcal{F}(x,y)\|_\infty}$, with $\hat{\mathcal{F}}_0$ being the first Fourier coefficient of $\mathcal{F}(x,y)$.

In the following examples, we compare the performance of the Laplacian, MGMs and circulant preconditioners. Precisely,
\begin{itemize}
	\item Like in the 1D case, here we consider the 2D Laplacian preconditioner ${P_{2}}=I_{M}-({P}_{x,M}+{P}_{y,M})$, where
	\begin{align*}
		%\begin{cases}
		&{P}_{x,M}=r_{1}\big(\mbox{C}_{l}(I_{M_{2}}\otimes {L}_{M_{1}})+\mbox{C}_{r}(I_{M_{2}}\otimes ({L}_{M_{1}})^{T})\big),\nonumber\\
		&{P}_{y,M}=r_{2}\big(\mbox{E}_{l}({L}_{M_{2}}\otimes I_{M_{1}})+\mbox{E}_{r}(({L}_{M_{y}})^{T}\otimes I_{M_{1}} )\big),\nonumber
		%	&{P_{2}}=I_{M}-({P}_{x,M}+{P}_{y,M}),
		%\end{cases}
	\end{align*}
	with ${L}_{M}=B^{(2,2)}_{M,0}$ (when $\gamma_{3}=0$) being the Laplacian matrix. Note that the structure of the coefficient matrix is not preserved as we do not take into account the advection terms in $\mathcal{M}_{x,M}$ and $\mathcal{M}_{y,M}$. \\
	$P_{2}^\nu$ denotes the inversion of ${P_{2}}$ through $\nu$ iterations of V$(0,1)$ with Galerkin approach, weighted Jacobi as smoother with $\omega=\frac{4}{5}$ and bilinear interpolation as grid transfer operator. In the following Tables \ref{tab:T21}-\ref{tab:T22}, we also report the results provided by the exact inversion of $P_2$, denoted by $\widetilde{P}_{2}$.
	%${P_{2}}$ is a pentadiagonal matrix and, due to fill-in, the direct solution often leads to dense matrix. To keep a low computational cost of the preconditioning iteration, we approximate ${P_{2}}^{-1}$ through one iteration of V$(0,1)$ with Galerkin approach, weighted Jacobi as smoother and bilinear interpolation as grid transfer operator. In the following Tables, we also report the results provided by the exact inversion of $P_2$, denoted by $\widetilde{P}_{2}$.
	%, where the preconditioner ${P}_{2}$ is applied as a direct solution to the resulting linear system.
	\item We use V$(1,1)$ as GMRES preconditioner with both Galerkin and geometric approaches, respectively denoted by $\widetilde{P}$V$(1,1)$ and $P$V$(1,1)$. In both cases, the grid transfer operator is bilinear interpolation, the weight of Jacobi is $\omega^\star$ computed as in \eqref{eq:w2D}, and one iteration of V-cycle is performed to approximate the inverse of ${\mathcal{A}_{(\alpha,\beta),M}}$.
	\item The standard circulant preconditioner is defined as  
	\begin{align*}
		{P_{\mbox{C}}}=I_{M}-\Big[r_{1}{c^{+}}\Big(I_{M_{2}}\otimes (\mathcal{C}(B^{(2,\alpha)}_{M_{1},\lambda_{1}})+\mathcal{C}(B^{(2,\alpha)}_{M_{1},\lambda_{1}})^{T})\Big)+r_{2}{e^{+}}\Big( (\mathcal{C}(B^{(2,\beta)}_{M_{2},\lambda_{2}})+\mathcal{C}(B^{(2,\beta)}_{M_{2},\lambda_{2}})^{T})\otimes I_{M_{1}}\Big)\Big],
	\end{align*}
	where ${c^{+}}=\mbox{mean}({\bf{C}_{s}})$, ${e^{+}}=\mbox{mean}({\bf{E}_{s}})$ and $\mathcal{C}(B)$ is the Chan circulant approximation of the Toeplitz matrix $B\in\mathbb{R}^{N\times N}$. An important advantage of circulant preconditioning is that $\left(\mathcal{C}(B)\right)^{-1}$ can be computed exactly in $O(N\log N)$ operations through the FFT algorithm. Nevertheless, it is well known that multilevel circulant matrices cannot ensure superlinear convergence if used as preconditioner for multilevel Toeplitz matrices \cite{capizzano2000any,donatelli2018spectral}. %The endorsement is given in Tables 6 and 7, where the number of iterations increase as the matrix size increase.  
\end{itemize}

{\bf Example 4.} Let $\Omega=[0,2]\times [0,2],\ T=2$ and consider
\begin{align*}
	&\mbox{C}_{l}(x,y)=\Gamma(3-\alpha)(1+x)^\alpha(1+y)^2,\quad\mbox{C}_{r}(x,y)=\Gamma(3-\alpha)(3-x)^\alpha(3-y)^2,\nonumber\\
	&\mbox{E}_{l}(x,y)=\Gamma(3-\beta)(1+x)^2(1+y)^\beta,\quad\mbox{E}_{r}(x,y)=\Gamma(3-\beta)(3-x)^2(3-y)^\beta,
\end{align*}
while the source term $f(x,y,t)$ is retrieved from the exact solution
\begin{align*}
	u(x,y,t)=16e^{-t}x^2y^2\left(2-x\right)^2\left(2-y\right)^2.
\end{align*}

Fixed $\alpha=1.8,\ \beta=1.6$, Table \ref{tab:T21} shows the iterations to tolerance and CPU times for different GMRES preconditioners varying $\lambda_{1},\lambda_{2}$, and $N=M_1=M_2$.\\
The number of iterations provided by the Laplacian preconditioner $\widetilde{P}_{2}$ is always smaller than that of $P_2^1$, nevertheless it is computationally expensive when $N$ is large. A good compromise here is in the form of $P_2^2$, with two iterations of V-cycle, whose \mbox{CPU}-times are slightly lower with respect to $P_2^1$.\\
Regarding the V-cycle preconditioner applied to the coefficient matrix, the Galerkin approach is more efficient than the geometric one when $N<2^7$. For $N\geq 2^7$, the Galerkin approach becomes computationally expensive in comparison with the geometric multigrid.\\
In terms of CPU times, both $P_{2}^1$ and $P_{2}^{2}$ beat other solvers for large sized liner systems $(N=2^7)$. In any case, all preconditioners but the circulant one, show linear convergence against the matrix size.

%%%%%%%%%%%
\begin{table}
	\caption{Example 4 - average number of iterations to tolerance of preconditioned GMRES for $\alpha=1.8$,$~\beta=1.6$ with $\gamma_{3}=0.0235$, $M_1=M_2=N$, and $\omega^\star=0.8507$.}
	\begin{small}
		\setlength{\tabcolsep}{7pt}
		\begin{center}
			\begin{tabular}{c c c c c c c c c c c c c c c c c c c}
				\hline
				\hline
				\noalign{\vskip 1mm}  
				\multirow{1}{*} {}&
				\multicolumn{2}{c}{${}$} &
				\multicolumn{2}{c}{$\widetilde{{P}}_{\mbox{2}}$} &
				\multicolumn{2}{c}{${{P}}_{\mbox{2}}^1$} &
				\multicolumn{2}{c}{${P}^{2}_{\mbox{2}}$} &
				\multicolumn{2}{c}{$\widetilde{{P}}V(1,1)$} &
				\multicolumn{2}{c}{$PV(1,1)$} &
				\multicolumn{2}{c}{$P_C$} \\
				\cmidrule(r){4-5}\cmidrule(r){6-7}\cmidrule(r){8-9}\cmidrule(r){10-11}\cmidrule(r){12-13}\cmidrule(r){14-15}
				$\lambda_{1},\lambda_{2}$&$N$& $\mbox{GMRES}$ & ${\mbox{It}}$&${\mbox{T(s)}}$& ${\mbox{It}}$ &${\mbox{T(s)}}$&${\mbox{It}}$ &${\mbox{T(s)}}$& ${\mbox{It}}$& ${\mbox{T(s)}}$& ${\mbox{It}}$&${\mbox{T(s)}}$& ${\mbox{It}}$&${\mbox{T(s)}}$ \\ \hline
				%%%%%%%%%%%%%%%%%%%%%%%%%%%%%%%
				${}$&$2^4$   &${38 }$&${9 }$&${0.262}$&${14}$&${0.220}$&${11}$&${0.231}$&${8}$&${0.206}$&${8}$&${0.440}$&${16}$&${0.430}$\\
				${\lambda_{1}=0}$&$2^5$  &${74 }$&${9 }$&${1.758}$&${16}$&${0.461}$&${13}$&${0.482}$&${9}$&${0.620}$&${9}$&${1.249}$&${22}$&${1.452}$\\
				${\lambda_{2}=0}$&$2^6$  &${139}$&${9 }$&${14.63}$&${17}$&${2.388}$&${14}$&${2.220}$&${9}$&${5.144}$&${9}$&${5.251}$&${27}$&${7.791}$\\
				${}$&$2^7$   &${256}$&${10}$&${139.7}$&${18}$&${21.36}$&${15}$&${31.76}$&${8}$&${64.92}$&${9}$&${25.59}$&${35}$&${65.54}$\\
				\hline
				${}$&$2^4$   &${41 }$&${7}$&${0.248}$&${14}$&${0.253}$&${10}$&${0.271}$&${8}$&${0.221}$&${9}$&${0.453}$&${17}$&${0.441}$\\
				${\lambda_{1}=1}$&$2^5$   &${76 }$&${7}$&${1.458}$&${15}$&${0.481}$&${12}$&${0.508}$&${9}$&${0.609}$&${9}$&${1.330}$&${21}$&${1.478}$\\
				${\lambda_{2}=1}$&$2^6$   &${139}$&${7}$&${11.90}$&${16}$&${2.407}$&${12}$&${2.186}$&${9}$&${5.364}$&${9}$&${5.323}$&${27}$&${7.336}$\\
				${}$&$2^7$   &${250}$&${8}$&${115.6}$&${16}$&${19.06}$&${13}$&${18.74}$&${8}$&${65.41}$&${9}$&${25.42}$&${34}$&${62.49}$\\
				\hline
				${}$&$2^4$   &${42 }$&${7}$&${0.261}$&${14}$&${0.226}$&${10}$&${0.229}$&${8}$&${0.231}$&${12}$&${0.532}$&${17}$&${0.450}$\\
				${\lambda_{1}=5}$&$2^5$   &${81 }$&${6}$&${1.314}$&${15}$&${0.488}$&${10}$&${0.466}$&${9}$&${0.647}$&${13}$&${1.777}$&${21}$&${1.468}$\\
				${\lambda_{2}=5}$&$2^6$   &${146}$&${6}$&${10.86}$&${15}$&${2.306}$&${11}$&${2.238}$&${9}$&${5.893}$&${11}$&${6.271}$&${26}$&${7.341}$\\
				${-}$&$2^7$  &${243}$&${6}$&${94.36}$&${14}$&${17.49}$&${11}$&${16.76}$&${8}$&${66.74}$&${10}$&${18.55}$&${33}$&${61.44}$\\
				\hline
			\end{tabular}
		\end{center}
	\end{small}
	\label{tab:T21}
\end{table}
%%%%%%%%%%%
{\bf Example 5.} This example is taken from \cite{yu2017third}. We consider the TFDE in equation \eqref{eq:Ex6} with diffusion coefficients ${D}_{l}={E}_{l}=1$, ${D}_{r}={E}_{r}=0$, spatial domain $\Omega=[\,0,1]\times[0,1]$, and final time $T=1$.
%The source term is 
% \begin{align*}
% 	f(x,y,t)=&(\lambda_{1}^{\alpha}+\lambda_{2}^{\beta}-1)u+\alpha\lambda_{1}^{\alpha-1}u_{x}+\beta\lambda_{2}^{\beta-1}u_{y}-{\rm{e}}^{-t-\lambda_{1}x-\lambda_{2}y}\Big[y^{4}(1-y)x^{4-\alpha}\frac{\Gamma(5)}{\Gamma(5-\alpha)}\Big(1-\frac{5}{5-\alpha}x\Big)+\nonumber\\ &x^{4}(1-x)y^{4-\beta}\frac{\Gamma(5)}{\Gamma(5-\beta)}\Big(1-\frac{5}{5-\beta}y\Big)\Big].
% \end{align*}
The forcing term is built from the exact solution 
\begin{align*}
	u(x,y,t)={\rm{e}}^{-t-\lambda_{1}x-\lambda_{2}y}x^4y^4\left(1-x\right)\left(1-y\right),
\end{align*}

In this case, we have a strong anisotropy, since in both dimensions the right diffusion coefficient is set to $0$. This leads to a strongly non-symmetric coefficient matrix. In order to apply the multigrid, we recover some symmetry by taking $\alpha,\beta\approx 2$, since in that case the remaining left tempered fractional operator tends to a Laplacian matrix.

Fixed $\alpha=1.8,\beta=1.6$, Table \ref{tab:T22} reports the iterations to tolerance of our preconditioners varying $\lambda_1,\lambda_2$ and the size $N=M_1=M_2$. Regarding the Laplacian preconditioner, the same comments as in Example 4 apply. In the case of $\widetilde{P}V(1,1)$ and $PV(1,1)$, here we note a decrease in iterations yield by the geometric approach with respect to the previous example when $\lambda_1=\lambda_2=5$.\\
Even in this case, all preconditioners but the circulant one, show linear convergence against the matrix size, therefore MGMs are shown to be robust preconditioners even in this anisotropic case.

%%%%%%%%%%
\begin{table}
	\caption{Example 5 - average number of iterations to tolerance of preconditioned GMRES for $\alpha=1.8,\beta=1.6$ with $\gamma_{3}=0.0235$, $M_1=M_2=N$, and $\omega^\star=0.8507$.}
	\begin{small}
		\setlength{\tabcolsep}{7pt}
		\begin{center}
			\begin{tabular}{c c c c c c c c c c c c c c c c c c c c}
				\hline
				\hline
				\noalign{\vskip 1mm}
				\multirow{1}{*} {}&
				\multicolumn{2}{c}{${}$} &
				\multicolumn{2}{c}{$\widetilde{{P}}_{\mbox{2}}$} &
				\multicolumn{2}{c}{${{P}}_{\mbox{2}}$} &
				\multicolumn{2}{c}{${P}^{2}_{\mbox{2}}$} &
				\multicolumn{2}{c}{$\widetilde{{P}}V(1,1)$} &
				\multicolumn{2}{c}{${{P}}V(1,1)$} &
				\multicolumn{2}{c}{$P_C$} &
				\multicolumn{1}{c}{${}$} \\
				\cmidrule(r){4-5}\cmidrule(r){6-7}\cmidrule(r){8-9}\cmidrule(r){10-11}\cmidrule(r){12-13}\cmidrule(r){14-15}
				$\lambda_{i's}$&$N$& ${\mbox{GMRES}}$& ${\mbox{It}}$&${\mbox{T(s)}}$& ${\mbox{It}}$ &${\mbox{T(s)}}$&${\mbox{It}}$ &${\mbox{T(s)}}$&${\mbox{It}}$&${\mbox{T(s)}}$&${\mbox{It}}$&${\mbox{T(s)}}$& ${\mbox{It}}$&${\mbox{T(s)}}$\\ \hline
				%%%%%%%%%%%%%%%%%%%%%%%%%%%%%%%
				${}$&$2^4$   &${38}$&${10}$&${0.182}$&${14}$&${0.229}$&${12}$&${0.233}$&${8}$&${0.230}$&${8}$&${0.407}$&${13}$&${0.459}$\\
				${\lambda_{1}=0}$&$2^5$  &${60}$&${10}$&${0.690}$&${15}$&${0.473}$&${12}$&${0.478}$&${8}$&${0.502}$&${9}$&${1.052}$&${15}$&${1.151}$\\
				${\lambda_{2}=0}$&$2^6$  &${90}$&${9 }$&${5.084}$&${15}$&${2.220}$&${13}$&${2.208}$&${8}$&${3.371}$&${9}$&${3.616}$&${17}$&${4.618}$\\
				${}$&$2^7$   &${121}$&${8}$&${66.07}$&${15}$&${18.44}$&${13}$&${18.47}$&${8}$&${54.51}$&${9}$&${15.99}$&${19}$&${25.58}$\\
				\hline
				${}$&$2^4$   &${35 }$&${8}$&${0.170}$&${13}$&${0.225}$&${10}$&${0.237}$&${8}$&${0.211}$&${8}$&${0.416}$&${13}$&${0.461}$\\
				${\lambda_{1}=1}$&$2^5$   &${54 }$&${7}$&${0.558}$&${14}$&${0.410}$&${11}$&${0.433}$&${8}$&${0.488}$&${8}$&${0.980}$&${15}$&${1.132}$\\
				${\lambda_{2}=1}$&$2^6$   &${78 }$&${7}$&${4.214}$&${14}$&${1.973}$&${11}$&${1.929}$&${8}$&${3.194}$&${8}$&${3.292}$&${17}$&${4.595}$\\
				${}$&$2^7$   &${105}$&${7}$&${40.46}$&${14}$&${17.23}$&${11}$&${16.81}$&${7}$&${49.02}$&${8}$&${14.88}$&${18}$&${24.26}$\\
				\hline
				${}$&$2^4$  &${30}$&${7}$&${0.173}$&${13}$&${0.213}$&${9}$&${0.224}$&${7}$&${0.204}$&${8}$&${0.387}$&${12}$&${0.392}$\\
				${\lambda_{1}=5}$&$2^5$  &${43}$&${7}$&${0.561}$&${12}$&${0.394}$&${9}$&${0.419}$&${7}$&${0.463}$&${8}$&${0.960}$&${13}$&${0.984}$\\
				${\lambda_{2}=5}$&$2^6$  &${59}$&${7}$&${4.536}$&${12}$&${1.948}$&${9}$&${1.821}$&${7}$&${3.171}$&${8}$&${3.431}$&${14}$&${3.793}$\\
				${}$&$2^7$  &${79}$&${6}$&${63.02}$&${12}$&${15.73}$&${10}$&${14.63}$&${7}$&${50.43}$&${8}$&${16.42}$&${15}$&${23.28}$\\
				\hline
			\end{tabular}
		\end{center}
	\end{small}
	\label{tab:T22}
\end{table}

%%%%%%%%%%
\section{Conclusions}\label{sect:concl}
In this paper, we have investigated multigrid methods for time-dependent tempered fractional diffusion equations. After providing a symbol-based detailed spectral analysis of the coefficient matrix, we have exploited such information to prove the stability of the time stepping CN-TWSGD scheme.
Moreover, we have extended the theoretical results in \cite{bu2021multigrid} regarding the convergence of multigrid methods, in case of constant and equal diffusion coefficients, to a wider interval of fractional derivatives, i.e., from $\alpha\in(1.26,1.71)$ to the entire interval $\alpha\in(1,2)$. We have also expanded the interval of the Jacobi relaxation parameter from $\omega\in(0,0.5]$ to $\omega\in(0,\xi]$, where $\xi$ depends on the symbol of the coefficient matrix and is usually greater than $0.5$. We have further exploited the symbol to provide a cheap to be computed suitable weight $\omega^\star$ for the multigrid method.

In the numerical results section, we reported examples that support our theoretical results. We have shown that the estimated suitable weight for Jacobi $\omega^\star$ often leads to the smallest iteration count to tolerance. Furthermore, numerical results with non-constant and non-equal diffusion coefficients show that our multigrid-based solver is robust, even when violating the constraints imposed in the theoretical settings. Finally, tests in 1D and 2D show that when the fractional derivatives are close to 2, the Laplacian preconditioner allows fast convergence like in the case of non-tempered fractional derivatives \cite{donatelli2016spectral}.

	\section*{Acknowledgments}
	This research was supported by the GNCS-INDAM (Italy) and the EuroHPC TIME-X project n. 955701.

\end{document}